\documentclass[11pt]{amsart}

\usepackage{verbatim, amssymb,hyperref}

\usepackage{accents}
\usepackage{color}
\usepackage{soul}
\usepackage{bbm}

\newcommand{\black}{\color[rgb]{0,0,0}}
\newcommand{\red}{\color[rgb]{1,0,0}}
\newcommand{\blue}{\color[rgb]{0,0,1}}

\newcommand{\eps}{\varepsilon}

\newcommand{\cF}{\mathcal F}

\newcommand{\cH}{\mathcal H}

\newcommand{\cU}{\mathcal U}
\newcommand{\cO}{\mathcal O}

\newcommand{\var}{\mathrm{Var}}

\newcommand{\cost}{\operatorname{cost}}

\newcommand{\sfrac}[2]{\mbox{$\frac{#1}{#2}$}}

\newcommand{\II}{\mathbb I}

\newcommand{\cov}{\mathrm{cov}}

\newcommand{\sig}{\sigma}

\newcommand{\1}{1\hspace{-0.098cm}\mathrm{l}}

\renewcommand{\P}{{\mathbb P}}
\renewcommand{\Pr}{{\mathbb P}}

\newcommand{\N}{{\mathbb N}}

\newcommand{\E}{{\mathbb E}}

\newcommand{\R}{{\mathbb R}}

\newcommand{\ID}{{\mathbb D}}

\newcommand{\pr}{\operatorname{pr}}

\theoremstyle{plain}
\newtheorem{theorem}{Theorem}[section]
\newtheorem{prop}[theorem]{Proposition}
\newtheorem{lemma}[theorem]{Lemma}
\newtheorem{cor}[theorem]{Corollary}
\newtheorem{defi}[theorem]{Definition}
\newtheorem{ext}[theorem]{Extension}
\newtheorem{algo}[theorem]{Algorithm}

\theoremstyle{definition}
\newtheorem{rem}[theorem]{Remark}

\begin{document}

\title[Multilevel stochastic approximation]%
{General multilevel adaptations for stochastic approximation algorithms}

\author[]
{Steffen Dereich}
\address{Steffen Dereich\\
Institut f\"ur Mathematische Stochastik\\
Fachbereich 10: Mathematik und Informatik\\
Westf\"alische Wilhelms-Universit\"at M\"unster\\
Orl\'eans-Ring 10\\
48149 M\"unster\\
Germany}
\email{steffen.dereich@wwu.de}


\thanks{Funded by the Deutsche Forschungsgemeinschaft (DFG, German Research Foundation) under Germany's Excellence Strategy EXC 2044 –390685587, Mathematics Münster: Dynamics–Geometry–Structure.}

\begin{abstract}
In this article we establish central limit theorems for multilevel  Polyak-Ruppert-averaged stochastic approximation schemes. We work under very mild technical assumptions and consider the slow regime in wich typical errors decay like $N^{-\delta}$ with $\delta\in(0,\frac 12)$ and the critical regime in which errors decay of order $N^{-1/2}\sqrt{\log N}$ in the runtime $N$ of the algorithm.
\end{abstract}

\maketitle

\section{Introduction}\label{s1}

Let  $D\subset \R^d$ be a closed set and $U$ a random variable on a  probability space $(\Omega,\cF,\Pr)$
with values in a set $\mathcal U$ equipped with some $\sigma$-field.
We consider stochastic approximation algorithms for the computation of zeroes of functions $f\colon D\to\R^d$ 
of the form
\[
f(\theta)=\E[F(\theta,U)], 
\]
where $F\colon D\times \mathcal U\to \R^d$ is a product measurable function such that all expectations are well-defined.  

Stochastic approximation methods form  a popular class of optimization algorithms that attracted significant interest in previous years. The original idea stems from an article by Robbins and Monro \cite{RM51}. Later it was found by Ruppert \cite{Rup82} and Polyak \cite{Pol90,PJ92}  that averaging  improves the order of convergence in the case of slower decaying stepwidths.
Following these original papers a variety of results were derived and we refer the reader to the monographs by  \cite{BMP90,Duf96,KY03} for more details.

In this article we focus on the case where the random variable $F(\theta,U)$ cannot be  simulated directly so that one has to work with appropriate approximations in numerical simulations. In previous years the multilevel paradigm \cite{Gil08} has proven to be a very efficient tool in the numerical computation of expectations. This advantage prevails for multilevel stochastic approximation algorithms \cite{Fri15, DerMue18}. In this article we derive central limit theorems under general assumptions on the underlying approximations. In particular, we investigate the role of the chosen parameters.

The article has three main results. We prove a new central limit theorem for Ruppert-Polyak averaging which extends previous research of Mokkadem and Pelletier~\cite{MP11}. Here we do not impose regular variation assumptions on the involved key quantites. Furthermore, we establish a new central limit theorem  for an averaged multilevel stochastic approximation scheme in the slow regime in wich typical errors decay like $N^{-\delta}$  in the runtime $N$ of the algorithm with $\delta\in(0,\frac 12)$, see Section~\ref{sec_1_2-1}. Finally we provide a central limit theorem for the  critical regime in which errors decay like $N^{-1/2}\sqrt{\log N}$,  see Section~\ref{sec_1_3-1}.
The remaining fast regime in which errors decay like $N^{-1/2}$ is omitted in our analysis. In this setting unbiased multilevel estimates of~\cite{McL11} and~\cite{RG15} can be used  in which case a central limit theorem immediately follows from classical theory.


We will use Landau notation and write $a_n=\cO(b_n)$ for two sequences $(a_n)$ and $(b_n)$ if $b_n$ is eventually strictly positive and $\limsup |a_n|/b_n<\infty$.  If, additionally, $\lim |a_n|/b_n=0$, we write $a_n=o(b_n)$. Furthermore, for two eventually strictly positive sequences $(a_n)$ and $(b_n)$ we write $a_n\sim b_n$ if $\lim a_n/b_n=1$.

\subsection{CLT for Ruppert-Polyak averaging}\label{sec_1_1}

Our analysis is based on a new central limit theorem for Ruppert-Polyak averaging. Compared to previous research we will not assume that the key quantities are regularly varying.
In the following, $(\Omega,\cF,(\cF_n)_{n\in\N_0},\P)$ denotes a filtered probability space and the notions of martingales and adapted processes always refer to the latter space.

Let us introduce the central dynamical system. 
Let $f:D\to\R^d$ be a measurable function on a   closed set $D\subset \R^d$. 
We consider an adapted $D$-valued dynamical system $(\theta_n)_{n\in\N_0}$ satisfying for all $n\in\N$ \begin{equation}\label{dynsys2}
\theta_{n}=\Pi\bigl(\theta_{n-1}+\gamma_n \bigl( f(\theta_{n-1})+  R_n + D_{n}\bigr)\bigr),
\end{equation}
where $\theta_0\in D$ is a fixed deterministic starting value, 
\begin{enumerate}
\item[(0)]  $\Pi:\R^d\to D$ is a  measurable function with $\Pi|_D$ being the identity on $D$ and $\Pi|_{D^c}$ mapping into $\partial D$, the \emph{projection},
\item[(I)]  $(R_n)_{n\in\N}$ is an adapted  process, the \emph{remainder/bias},
\item[(II)] $(D_n)_{n\in\N}$ is a sequence of  \emph{martingale differences}, 
\item[(III)] $(\gamma_n)_{n\in\N}$ is a sequence of strictly positive reals, the \emph{step-widths}.
\end{enumerate}

Stochastic approximation algorithms may be applied to numerically estimate $L$-contracting zeroes of the function $f$.

\begin{defi}
For $L\geq  0$ a zero $\theta^*\in \mathrm{interiour}(D)$  of $f$ is called \emph{$L$-contracting zero} if  there is a matrix $H\in\R^{d\times d}$ with
 $$
 \sup\{\mathrm{Re}(\lambda): \lambda\text{ e.v.\ of }H\}< -L
 $$
 and
$$
f(\theta)=H(\theta -\theta^*)+ o(|\theta-\theta^*|)
$$
as $\theta\to \theta^*$. We say that $\theta^*$ is a contracting zero of $f$ if it is $0$-contracting.
\end{defi}



As observed by Ruppert and Polyak averaging may accelerate  stochastic approximation algorithms. Given a sequence $(b_n)_{n\in\N}$ of strictly positive reals we call
\begin{align}\label{eq_RupPol}
\bar \theta_n= \frac 1{\bar b_{n}} \sum_{k=1}^n b_k\theta_k\qquad (n\in\N)
\end{align}
with  $\bar b_n=\sum_{k=1}^n b_k$ the $(b_n)$-averaged version of $(\theta_n)$.

\begin{theorem}\label{theo_RP_main}
We consider the dynamical system $(\theta_n)_{n\in\N}$ introduced in~(\ref{dynsys2}) and assume that $\theta^*\in D$ is a $L$-contracting zero of $f$ for a $L>0$. Let further $\lambda \in (0,1]$ and  $(\delta_n^\mathrm{bias})_{n\in\N}$, $(\delta_n^\mathrm{diff})_{n\in\N}$ and $(\delta_n)_{n\in\N}$  be  sequences of strictly positive reals, set
$$
\eps_n^\mathrm{bias}=\frac1{\bar b_n} \sum_{k=1}^n b_k \delta_k^\mathrm{bias}\text{ \ and \ }\eps_n^\mathrm{diff}=\frac1{\bar b_n} \sqrt{\sum_{k=1}^n (b_k \delta_k^\mathrm{diff})^2}
$$
and assume that  the following assumptions are satisfied:
\begin{enumerate}
\item[(A)]\emph{Regularity of $(\gamma_n)$, $(b_n)$ and $(\delta_n)$.} $(\gamma_n)_{n\in\N}$ is decreasing with limit zero and $$\lim_{n\to\infty} n\gamma_n=\infty, \ \limsup_{n\to\infty}\frac{\gamma_{n-1}}{\gamma_{n}}<\infty, 
$$ 
$$\limsup_{n\to\infty} \frac 1{\gamma_n} \bigl(1- \frac {\delta_n}{\delta_{n-1}}\bigr) \leq L  \text{ \  and \ }\lim_{n\to\infty} \frac 1{\gamma_n} \Bigl( \frac{b_{n+1}\gamma_n}{b_n\gamma_{n+1}}-1\Bigr)= 0 .
$$
\item[(B)] \emph{Regularity of $f$ in $\theta^*$.} As $\theta\to\theta^*$,
$$
|f(\theta)-H(\theta-\theta^*)|= \mathcal O(|\theta-\theta^*|^{1+\lambda}) .
$$
\item[(C)] \emph{Tale estimates for $(\delta_n^\mathrm{bias})$ and $(\delta_n^\mathrm{diff})$.} One has
 $$
 \sum_{k=1}^\infty b_k \delta_k^\mathrm{bias}=\infty \text{ \ and \  }\sum_{k=1}^\infty (b_k\delta_k^\mathrm{diff})^2=\infty
 $$
and  for arbitrary $\N$-valued sequence $(L(n))_{n\in\N}$ with $L(n)\leq n$ and $n-L(n)=o(n)$
 $$
 \lim_{n\to\infty} \frac{\sum_{k=L(n)+1}^n b_k\delta_k^\mathrm{bias}}{\sum_{k=1}^n b_k\delta_k^\mathrm{bias}}=0 \text{ \ and \ } \lim_{n\to\infty} \frac{\sum_{k=L(n)+1}^n (b_k\delta_k^\mathrm{diff})^2}{\sum_{k=1}^n (b_k\delta_k^\mathrm{diff})^2}=0.
 $$
 \item[(D)]\emph{Assumptions on the bias $(R_n)$.} One has on $\{\theta_n\to\theta^*\}$, up to nullsets,
 $$
(\delta_n^{\mathrm{bias}})^{-1} R_n\to \mu\in\R^d
$$
 and for a $\eps'>0$
 $$
 \limsup_{n\to\infty} \delta_n^{-1} \E[\1_{\{\theta_{n-1}\in B(\theta^*,\eps')\}} |R_n|^2]^{1/2}<\infty.
 $$
 \item[(E)]\emph{Assumptions on the diffusivity $(D_n)$.} On $\{\theta_n\to\theta^*\}$,
 $$
(\delta_n^{\mathrm{diff}})^{-2} \cov(D_n|\cF_{n-1})\to \Gamma,\text{ \ up to nullsets,}
$$
$$ \lim_{n\to\infty} (\eps_n^\mathrm{diff})^{-2} \sum_{l=1}^n \frac {b_l^2}{\bar b_n^2} \E[\1_{\{|D_l|>\eps\bar b_n \eps_n^\mathrm{diff}/b_l\}}|D_l|^2|\cF_{l-1}]=0, \text{ \ in probability},
$$
and for an $\eps'>0$ one has
$$
\limsup_{n\to\infty} \bigl( \frac{\delta_n}{\sqrt{\gamma_n}}\bigr)^{-1} \E[\1_{ B(\theta^*,\eps')}(\theta_{n-1}) |D_n|^2]^{1/2}<\infty.
$$
\item[(F)] \emph{Negligibility of linearization error.} The term
\begin{align}\label{eq74598}
\frac 1{\bar b_n} \sum_{l=1}^n  b_l \delta_l^{1+\lambda} 
\end{align}
is of order $o(\eps_n^\mathrm{bias})$ or $o(\eps_n^\mathrm{diff})$.
\end{enumerate}
Then there exists $(\vartheta_n)_{n\in\N}$ such that on $\{\theta_n\to \theta^*\}$
$$
(\eps_n^\mathrm{bias})^{-1} (\vartheta_n-\theta^*) = - H^{-1} \mu\text{, \ in probability}
$$
and
$$
 (\eps_n^\mathrm{diff})^{-1}  (\bar \theta_n-  \vartheta_n) \Rightarrow \mathcal N(0,H^{-1} \Gamma (H^{-1})^\dagger).
$$
\end{theorem}

The introduction of $(\vartheta_n)$  has technical reasons. In the case where one of the terms~$\eps_n^\mathrm{bias}$ or $\eps_n^\mathrm{var}$ asymptotically dominates the other, one can phrase a central limit theorem without the use of $(\vartheta_n)$.

\begin{cor}\label{cor_1}
Assume that all assumptions of Theorem~\ref{theo_RP_main} are satisfied.
\begin{enumerate}
\item If $\eps_n^\mathrm{bias}=\mathcal O(\eps_n^\mathrm{diff})$, then one has on $\{\theta_n\to\theta^*\}$
$$
(\eps_n^\mathrm{diff})^{-1} (\bar \theta_n -\theta^*+\eps_n^\mathrm{bias} H^{-1} \mu)\Rightarrow \mathcal N(0, H^{-1}\Gamma (H^{-1})^\dagger).
$$
\item If $\eps_n^\mathrm{bias}=o(\eps_n^\mathrm{diff})$, then one has on $\{\theta_n\to\theta^*\}$
$$
(\eps_n^\mathrm{diff})^{-1} (\bar \theta_n -\theta^*)\Rightarrow \mathcal N(0, H^{-1}\Gamma (H^{-1})^\dagger).
$$
\item If $\eps_n^\mathrm{diff}=o(\eps_n^\mathrm{bias})$, then one has on $\{\theta_n\to\theta^*\}$
$$
(\eps_n^\mathrm{bias})^{-1} (\bar \theta_n -\theta^*)\to -H^{-1}\mu,\text{ \ in probability.}
$$
\end{enumerate}
\end{cor}
%
%

\begin{rem}[Discussion of the assumptions]
The crucial quantities  in the CLT (Theorem~\ref{theo_RP_main}) are $(\delta^{\mathrm{diff}}_n)$ and $(\delta^{\mathrm{bias}}_n)$. These control  the contribution of the martingale differences $(D_n)$ and the adapted term $(R_n)$ which is made explicit in (D) and (E). Beyond the latter assumptions we only impose very weak assumptions on these quantities in (C). The role of $(\delta_n)$ is to control a technical term appearing in our computations. It appears in assumptions (D) and (E). Further with the assumed regularity of $f$, i.e.\ assumption (B), the proof uses  that (F) implies negligibility of a certain linearisation error. Additionally, (A) imposes a technical assumption on $(\delta_n)$. In the context of Ruppert-Polyak averaging, one typically is in a setting where $(\delta_n)$ can be chosen such that $\frac1{\gamma_n}(1-\frac {\delta_n}{\delta_{n-1}})$ converges to zero. In that case the strength of the contractivity $L$ does not need to be known. 
The further assumptions imposed in (A) are very weak and essentially only exlclude the case where $(\gamma_n)$ is of order $(1/n)$.
\end{rem}

In the proof we follow the approach that was taken in Sacks~\cite{Sacks58} for the first time to prove distributional limit theorems for  stochastic approximation schemes. First, one derives an a priori bound for the order of convergence of $(\theta_n)$ to an $L$-contracting zero $\theta^*$. Second, one derives limit theorems for a related linearised system. Finally, one shows with the help of the a priori bound that  the difference between the  original  dynamical system and a linearised version is asymptotically negligible and that the results of the second part carry over to the Ruppert-Polyak average of the original stochastic approximation  scheme. 
For convenience of the reader the proof is fully self-contained. Our proofs are based on a classical  $L^2$-bound which is provided 
 in the appendix for the convenience of the reader, see Theorem~\ref{thm1}. The  linearised scheme (the second step) is analysed in Section~\ref{sec_RP_lin}. We stress that the respective results do not require regular variation assumptions as is standard in previous research. 
The proof of Theorem~\ref{theo_RP_main} will finally be carried out in Section~\ref{proof_main_SA}.



\subsection{Averaged multilevel stochastic approximation in the slow regime}\label{sec_1_2-1}

The introduction of multilevel  stochastic approximation algorithms is based on approximations $F_k:D\times \cU\to \R^d$ $(k\in\N)$ for $F$.

\begin{defi}\label{def_app1} For every $k\in \N$ let
\[
F_k\colon D\times \cU\to \R^d
\]
be a product measurable mapping such that for every $\theta\in D$, $F_k(\theta,U)$ is integrable.
Furthermore, let  $\theta^*\in D$, $\alpha,\beta>0$, $M>1$, $\mu\in\R^d$ and $\Gamma\in \R^{d\times d}$.
We say that the mappings $F_1,\dots$ are \emph{approximations for $F$ of order $(\alpha,\beta)$ with asymptotic bias $\mu$ and variance $\Gamma$ around $\theta^*$ on scale $M$} if the following holds: 
\begin{align}\label{as_bias1_1}\tag{BIAS}
\lim_{\theta\to\theta^*,k\to\infty} M^{\alpha k} (\E[F_k(\theta, U)]-f(\theta))= \mu,
\end{align}
\begin{align}\label{as_var1_1}\tag{VAR}
\lim_{\theta\to \theta^*, k\to \infty} M^{\beta k} \cov( F_k(\theta, U)-F_{k-1}(\theta,U))=\Gamma
\end{align}
and there 
exists $q>2$, a neighbourhood $D_{\theta^*}$ of $\theta^*$ and a constant $C_{\mathrm{TAIL}}$ such that for all $\theta\in D_{\theta^*}$
\begin{align}\label{as_tail_1}\tag{TAIL}
\E[\|F_{k}(\theta, U)-F_{k-1}(\theta,U)- \E[F_k(\theta, U)-F_{k-1}(\theta,U)]\|^{q}]^{2/q} \leq C_{\mathrm{TAIL}} \,  M^{-\beta k}.
\end{align}
\end{defi}

In the following we will associate a simulation of $F_k(\theta, U)-F_{k-1}(\theta,U)$ with a value $C_k(\theta)\in(0,\infty)$, the \emph{cost} of one simulation.
\begin{defi}

Let $\theta^*\in D$, $M>1$ and $\kappa_C>0$. We call a family $(C_k:k\in\N)$ of measurable functions $C_k:D\to (0,\infty)$ \emph{cost functions of scale $M$ with contant $\kappa_C$ around $\theta^*$}, if the following holds 
\begin{align}\label{as_cost_2}\tag{COST}
\lim_{\theta\to\theta^*,k\to\infty} M^{-k} C_k(\theta)= \kappa_C.
\end{align}
and there  exists a neighbourhood $D_{\theta^*}$ of $\theta^*$  and a constant $C_\mathrm{COST}$  such that for  $\theta\in D_{\theta^*}$ and $k\in\N$
\begin{align}\label{as_cost2_2}\tag{COST2}
C_k(\theta) \leq C_{\mathrm{COST}}\, M^k.
\end{align}
\end{defi}

\begin{algo}[Multilevel stochastic approximation I] \label{algo1} The definition depends on a family $F_1,\dots$ of approximations of $F$, the scale $M>0$ and on four further sequences:
\begin{itemize}\item a monotonically decreasing $(0,\infty)$-valued sequence  $(\gamma_n)_{n\in\N}$ (\emph{step-sizes})
\item a $(0,\infty)$-valued sequence $(K_n)_{n\in\N}$ (\emph{time budgets}), 
\item  a $\N$-valued sequence $(s_n)_{n\in\N}$ (\emph{accuracies}), 
\item a $(0,\infty)$-valued sequence $(b_n)_{n\in\N}$ (\emph{weights}).
\end{itemize}
For an $\cU$-valued sequence $\mathbbm u=(u_{k,l}:k,l=1,\dots)$,  $s\in\N$, $K>0$ and $\theta\in D$ we let
$$
Z(\mathbbm u;\theta, s,K)=\sum_{k=1}^{s} \frac 1{N_k(s,K)} \sum_{l=1}^{N_k(s,K)} (F_k(u_{k,l},\theta)- F_{k-1}(u_{k,l},\theta)),
$$
where  
$$
N_k(s,K)=  \Bigl\lceil \frac{K}{M^s} M^{\frac{\beta+1}2(s-k)}\Bigr\rceil.
$$
Let $\mathbb U=(U_{k,l})_{k,l=1,\dots}$ be a sequence of independent copies of $U$. 
For  given $\theta_0\in D$ we define   $(\theta_n:n\in\N_0)$ iteratively via 
$$
\theta_n=\Pi(\theta_{n-1} +\gamma_n Z(\mathbb U_n;\theta_{n-1}, s_n,K_n)),
$$
where  $(\mathbb U_n:n\in\N)$ is a sequence of independent copies of $\mathbb U$. Moreover, we let
$$
\bar \theta_n= \frac 1{\bar b_n}\sum_{k=1}^n b_k \theta_k
$$
and call $(\bar\theta_n)_{n\in\N_0}$  the \emph{output of the stochastic approximation algorithm with parameter tuple $((\gamma_n), (K_n),(s_n),(b_n))$}. If we are given measurable cost functionals
$$
C_k:D\to [0,\infty)
$$
we further associate the generation of $\bar\theta_n$ with the random cost
$$
\mathrm{cost}_n=\sum_{m=1}^n \sum_{k=1}^{s_m} N_k(s_m,K_m)\, C_k(\theta_{m-1}).
$$
\end{algo}

\begin{theorem}\label{theo:main_pol}
Let $\theta^*\in \mathrm{interiour}(D)$ be a  contracting zero of $f(\theta)=\E[F(\theta,U)]$ and suppose that there exist $H\in\R^{d\times d}$ and  $\lambda\in(0,1]$ with 
\begin{align}\label{as:dif}
|f(\theta)-H (\theta-\theta^*)|=\mathcal O(|\theta-\theta|^{1+\lambda}).
\end{align}
Let $M>1$, $\alpha,\beta>0$ with $\beta <2\alpha$ and $\beta<1$, and suppose that $F_1,\dots$ are approximations of $F$ of order $(\alpha,\beta)$ with asymptotic bias $\mu\in\R^d$ and variance $\Gamma\in\R^{d\times d}$ on scale $M>1$ around $\theta^*$.

Let $ \varphi, \rho$ and $\psi$ be reals with 
\begin{align}\label{eq:as_1}
\frac 1{2\alpha-\beta}<\varphi, \  \ \varphi+1< 2(\rho+1) \ \text{ and } \ \Bigl(1-\frac {2\lambda}{\lambda+1}(\varphi+1)\mathfrak r\Bigr)_+<\psi<1,
\end{align}
where 
$$\mathfrak r=\frac{\alpha}{2\alpha-\beta+1}.$$ 
In dependence on two further parameters $\kappa_K,\kappa_s>0$ we denote by $(\bar\theta_n)$ the averaged stochastic approximation scheme as in Algorithm~\ref{algo1} with parameters
$$
\gamma_n=n^{-\psi}, \ K_n=\kappa_K(\varphi+1)n^\varphi, \ {s_n}=\max\Bigl(\Bigl\lfloor \log_M  \kappa_s\, \bar K_n^{ \frac1{2\alpha-\beta+1}}\Bigr\rfloor,1\Bigr)\text{ \ and \ }b_n=n^\rho,
$$
where $\bar K_n=\sum_{k=1}^n K_k$.

We set 
$$
 \mathfrak r_1:=\frac {\rho+1}{\varphi+1}(2\alpha-\beta+1)\text{ \ and \ }\mathfrak r_2:=\bigl(2\frac{\rho+1}{\varphi+1}-1\bigr)(2\alpha-\beta+1),$$
let 
$$
\xi_n:=\log_M \kappa_s \bar K_n^{\frac 1{2\alpha-\beta+1}}- s_n
$$
and for $u,v\in\R$ with $u>0$ and $u-v>0$
$$\psi_{u,v}(z)=M^{-z(u+v)}\Bigl(\frac{M^u-1}{M^{u+v}-1} +M^{uz}-1\Bigr) \qquad (z\in[0,1]).
$$
Then one has on $\{\theta_n\to\theta^*\}$
$$
\frac 1{\eps_n^\mathrm{diff}} (\bar\theta_n-\theta^*+ \eps_n^\mathrm{bias}  H^{-1}\mu)\Rightarrow \mathcal N(0,H^{-1} \Gamma (H^{-1})^\dagger),
$$
where 
$$
\eps_n^\mathrm{bias}= \kappa_s^{-\alpha} \kappa_K^{-r}  \psi_{\mathfrak r_1,-\alpha}(\xi_n) \, n^{-(\varphi+1)\mathfrak r}
$$
and
$$
\eps_n^\mathrm{diff}=\frac{1}{\sqrt{1-M^{-\frac{1-\beta}2}}} \, \frac {\frac{\rho+1}{\varphi+1}}{\sqrt{2\frac{\rho+1}{\varphi+1}-1}}\, \kappa_s^{\frac{1-\beta}2}\kappa_K^{-\mathfrak r}\sqrt{\psi_{\mathfrak r_2,1-\beta}(\xi_n)}\,n^{-(\varphi+1)\mathfrak r}.
$$
Supposing further that $C_1,\dots$ are cost functions of scale $M$ with constant $\kappa_C>0$ around $\theta^*$ we get
in terms of $\mathrm{cost}_n$  on $\{\theta_n\to\theta^*\}$
$$
\eps_n^\mathrm{bias}\sim \kappa_C^{\mathfrak r} (1-M^{-\frac{1-\beta}2})^{-\mathfrak r}  \kappa_s^{-\alpha}  \psi_{\mathfrak r_1,-\alpha}(\xi_n)\,\mathrm{cost}_n^{-\mathfrak r}
$$
and
$$
\eps_n^\mathrm{diff}\sim \kappa_C^{\mathfrak r} (1-M^{-\frac{1-\beta}2})^{-(\mathfrak r+\frac 12)} \frac {\frac{\rho+1}{\varphi+1}}{\sqrt{2\frac{\rho+1}{\varphi+1}-1}}  \,\kappa_s^{\frac{1-\beta}2}\sqrt{\psi_{\mathfrak r_2,1-\beta}(\xi_n)}\,\mathrm{cost}_n^{-\mathfrak r}.
$$
\end{theorem}

The proof is carried out in Section~\ref{sec4}.

\subsection{Averaged multilevel stochastic apparoximation in the critical regime}\label{sec_1_3-1}


The critical regime is the setting  where the  approximations are of order $(\alpha,1)$. It is the most prominent regime since it comprises (under appropriate assumptions) the case where $f$ is given as expectation of a payoff function applied to the final value of a stochastic differential equation and we mention \cite{Gil08,Fri15} as classical references for the multilevel treatment of the respective case.  In that setting we  derive a central limit theorem under weaker assumptions than in the slow regime and we replace Definition~\ref{def_app1} by the following one.

\begin{defi} For every $k\in\N$ let 
$$F_k:D\times \cU\to \R^d
$$
be a product measurable mapping such that for every $\theta\in D$, $F_k(\theta,U)$ is integrable. Furthermore, let $\theta^*\in D$, $\alpha>0$, $M>1$ and $\Gamma\in \R^{d\times d}$.
We say that the mappings $F_1,\dots$ are \emph{approximations for $F$ of order $(\alpha,1)$ with asymptotic variance $\Gamma$ around $\theta^*$ on scale $M$} if the following holds:  there exists a sequence $(\alpha_k)_{k\in\N}$ tending to $\alpha$ such that 
\begin{align}\label{as_bias1_1-2}\tag{BIAS$^*$}
\limsup_{\theta\to\theta^*,k\to\infty} M^{\alpha_k k} |\E[F_k(\theta, U)]-f(\theta)|< \infty,
\end{align}
\begin{align}\label{as_var1_1-2}\tag{VAR}
\lim_{\theta\to \theta^*, k\to \infty} M^{\beta k} \cov( F_k(\theta, U)-F_{k-1}(\theta,U))=\Gamma
\end{align}
and there 
exists a neighbourhood $D_{\theta^*}$ of $\theta$ and a constant $C_{\mathrm{TAIL}}$ such that for all $\theta\in D_{\theta^*}$
\begin{align}\label{as_tail_1-1}\tag{TAIL}
\E[\|F_{k}(\theta, U)-F_{k-1}(\theta,U)- \E[F_k(\theta, U)-F_{k-1}(\theta,U)]\|^{q}]^{2/q} \leq C_{\mathrm{TAIL}} \,  M^{-\beta k}.
\end{align}
\end{defi}
\begin{rem}\label{rem_asalt} Obviously,  assumption~(\ref{as_bias1_1-2}) is weaker than assumption~(\ref{as_bias1_1}). 
\end{rem}

We now consider the case $\beta=1$. Note that in this case
$$
N_k(s,K)=  \bigl\lceil K M^{-k}\bigr\rceil
$$
so that the parameter $s$ has no influence on the iteration numbers and we write $N_k(K)=N_k(s,K)$ in the following.

\begin{theorem}\label{theo:main_pol_2}
Let $\theta^*\in \mathrm{interiour}(D)$ be a  contracting zero of $f(\theta)=\E[F(\theta,U)]$ and suppose that there exist $H\in\R^{d\times d}$ and  $\lambda\in(0,1]$ with 
\begin{align}\label{as:dif}
|f(\theta)-H (\theta-\theta^*)|=\mathcal O(|\theta-\theta|^{1+\lambda}).
\end{align}
Let $M>1$, $\alpha>1/2$  and suppose that $F_1,\dots$ are approximations of $F$ of order $(\alpha,1)$ with asymptotic  variance $\Gamma\in\R^{d\times d}\backslash \{0\}$ on scale $M>1$ around $\theta^*$.

Let $ \varphi, \rho$ and $\psi$ be reals with 
\begin{align}\label{eq:as_1-2}
\frac 1{2\alpha-1}<\varphi, \  \  \varphi+1< 2(\rho+1) \ \text{ and } \ \Bigl(1-\frac {\lambda}{\lambda+1}(\varphi+1)\Bigr)_+<\psi<1,
\end{align}
let $(\alpha_k)_{k\in\N}$ be as in~(\ref{as_bias1_1-2}) and choose an increasing $\N$-valued  sequence $(s_n)_{n\in\N}$ with
\begin{align}\label{as_634}
M^{-\alpha_{s_n} s_n}=o (n^{-\frac{\varphi+1}2}\sqrt{\log n}) \text{ \ and \ }
 \alpha s_n\sim \log_M n^{\frac{\varphi+1}2}.
\end{align}
In dependence on two further parameters $\kappa_K,\kappa_s>0$ we denote by $(\bar\theta_n)_{n\in\N_0}$ the averaged stochastic approximation scheme as in Algorithm~\ref{algo1} with parameters
$$
\gamma_n=n^{-\psi}, \ K_n=\kappa_K(\varphi+1)n^\varphi, \ {s_n}\text{ \ and \ }b_n=n^\rho.
$$
Then one has on $\{\theta_n\to\theta^*\}$
$$
\frac 1{\eps_n^\mathrm{diff}} (\bar\theta_n-\theta^*)\Rightarrow \mathcal N(0,H^{-1} \Gamma (H^{-1})^\dagger),
$$
where 
$$ 
\eps_n^\mathrm{diff}=\frac 1{\sqrt{2\alpha\kappa_K}} \frac {\frac{\rho+1}{\varphi+1}}{\sqrt{2\frac {\rho+1}{\varphi+1}-1}} n^{-\frac{\varphi+1}2} \sqrt{\log_Mn^{\varphi+1}}.
$$
Supposing further that $C_1,\dots$ are cost functions of scale $M$ with constant $\kappa_C>0$ around $\theta^*$ we get
in terms of $\mathrm{cost}_n$  on $\{\theta_n\to\theta^*\}$
$$
\eps_n^\mathrm{diff}\sim \frac{\sqrt{\kappa_C}}{2\alpha} \frac {\frac{\rho+1}{\varphi+1}}{\sqrt{2\frac{\rho+1}{\varphi+1}-1}}  \,\frac{\log_M \mathrm{cost}_n}{\sqrt{\mathrm{cost}_n}} . 
$$
\end{theorem}
The proof is achieved in Section \ref{sec5}.

\begin{rem}[Optimal choice of parameters]
In the theorem only the choice of $\frac{\rho+1}{\varphi+1}$ affects the asymptotic efficiency. By assumption this value needs to be strictly bigger than $\frac 12$ and elementary analysis implies that the optimal choice is the value one. Consequently, we see the optimal speed of convergence if we choose $\varphi$ and $\psi$ according to~(\ref{eq:as_1-2}) and set $\rho=\varphi$ which is always an allowed choice.
\end{rem}
\begin{rem}[Availability of feasible $s_n$] In the theorem it is left open whether   sequences $(s_n)$ with property~(\ref{as_634}) do exist. This is indeed a consequence of the other assumptions. We set for $\alpha'\in(0,\alpha)$
$$
s_n^{(\alpha')}=\Bigl\lceil \frac 1{\alpha'} \log_M n^{\frac{\varphi+1}2}\Bigr\rceil.
$$
Note that for sufficiently large $n\in\N$, $\alpha_{s_n^{(\alpha')}}>\alpha'$ and for these $n$
$$
M^{-\alpha_{s_n^{(\alpha')}}  s_n^{(\alpha')}}\leq n^{-\frac{\varphi+1}2} .
$$
By a diagonalisation argument we can choose $\alpha'_n\to\alpha$ such that for $s_n=s_n^{(\alpha'_n)}$
$$
M^{-\alpha_{s_n}  s_n}\leq n^{-\frac{\varphi+1}2}=o\bigl(n^{-\frac{\varphi+1}2}\sqrt{\log n}\bigr).
$$
By construction one also has
$$
s_n=s_n^{(\alpha_n')}\sim \frac 1{\alpha'_n} \log_M n^{\frac{\varphi+1}2}\sim \frac 1\alpha \log_Mn^{\frac{\varphi+1}2}.
$$
\end{rem}

\section{The Ruppert-Polyak system for linear systems}\label{sec_RP_lin}

In this section, we treat the particular case where $f$ is a linear function.
In the following $H$ denotes a fixed contracting ${d\times d}$-matrix. We denote by $(\Upsilon_n)_{n\in\N}$ a sequence  of $\R^d$-valued random variables and let for $n\in\N$
\begin{align}\begin{split}\label{eq:lin}
\theta_n&= \theta_{n-1} +\gamma_n (H \theta_{n-1}+ \Upsilon_n)\\
\bar \theta_n&= \frac 1{\bar b_{n}} \bigl(\bar b_{n-1}   \bar \theta_{n-1}+ b_{n} \theta_n\bigr)
\end{split}\end{align}
where $(\gamma_n)_{n\in\N}$ and $(b_n)_{n\in\N}$ are sequences of strictly positive reals with the former sequence being monotonically decreasing and with  $\bar b_n=\sum_{k=1}^n b_k$. Our analysis will rely on the following assumptions:
\begin{enumerate}
\item[(L1)] $\exists L>0,\eps_0\in\N$  $\forall \eps\in[0,\eps_0]:$ $\|\1+\eps H\|\leq 1-\eps L$, 
\item[(L2)] $n\gamma_n\to\infty$,
\item[(L3)]  $\displaystyle{\frac{\gamma_{n}}{ \gamma_{n+1}} =\mathcal O(1) \text{ \  and \ } \frac{b_{n+1}\gamma_n}{b_n\gamma_{n+1}}= 1+o(\gamma_n).}$ 
\end{enumerate}

\begin{theorem}\label{theo:lin_main}Assume that properties (L1), (L2) and (L3) hold.
\begin{enumerate}\item[{[}I{]}] Suppose that  $(\delta_n)_{n\in\N}$ is a sequence of strictly positive reals such that up to nullsets, on $\{\theta_n\to0\}$,
$$
\delta_n^{-1} \Upsilon_n \to \mu
$$
and that for every $(L(n))_{n\in\N}$ with $L(n)\leq n$ and $n-L(n)=o(n)$ one has
\begin{align}\label{eq7845}
\sum_{k=1}^\infty b_k \delta_k=\infty\text{ \ and \ }
\lim_{n\to\infty}\frac{\sum_{k=L(n)+1}^n b_k \delta_k}{\sum_{k=1}^n b_k \delta_k}= 0.
\end{align}
Then up to nullsets on $\{\theta_n\to0\}$ one has 
$$
\lim_{n\to\infty} \Bigl(\frac  1{\bar b_n}\sum_{k=1}^n b_k \delta_k\Bigr)^{-1} \bar\theta_n  = -H^{-1}\mu.
$$
\item[{[}II{]}]
Suppose that $(\Upsilon_n)_{n\in\N}$ is a sequence of square integrable martingale differences  and that  $(\delta_n)_{n\in\N}$ is a sequence of strictly positive reals  such that up to nullsets, on $\{\theta_n\to0\}$,
$$
\lim_{n\to\infty} \frac 1{\delta^2_n} \cov(\Upsilon_n|\cF_{n-1})=\Gamma.
$$
Suppose further that for 
$$
\sigma_n=\frac 1{\bar b_n}\sqrt {\sum_{l=1}^n (b_l\delta_l)^2}
$$
and all $\eps>0$, on $\{\theta_n\to0\}$
$$ \lim_{n\to\infty} \sigma_n^{-2} \sum_{l=1}^n \frac {b_l^2}{\bar b_n^2} \E \bigl[\1_{\{\|\Upsilon_l\|>\frac{\eps\bar b_n \sigma_n}{b_l}\}}  \|\Upsilon_l\|^2\big| \cF_{l-1}\bigr]=0, \text{ in probability}.
$$
Additionally, we assume that for all sequences $(L(n))_{n\in\N}$ with $L(n)\leq n$ and $n-L(n)=o(n)$ one has
\begin{align}\label{eq974578}
\sum_{k=1}^\infty (b_k \delta_k)^2=\infty\text{ \ and \ }
\lim_{n\to\infty}\frac{\sum_{k=L(n)+1}^n (b_k \delta_k)^2}{\sum_{k=1}^n (b_k \delta_k)^2}= 0,
\end{align}
then it follows that, on $\{\theta_n\to0\}$,
$$
\sigma_n ^{-1} \bar \theta_n \Rightarrow \mathcal N(0, H^{-1}\Gamma (H^{-1})^\dagger).
$$
\item[{[}III{]}] Suppose that $(\eps_n)_{n\in\N}$ is a sequence of strictly positive reals such that there exists $n_0\in\N$ with
$$
\lim_{n\to\infty} \eps_n^{-1}  \bar b_n^{-1} \sum_{l=n_0+1}^n b_l \E[ \|\Upsilon_n\|] = 0.
$$
Then 
$$
\lim_{n\to\infty} \eps_n^{-1} \bar\theta_n  = 0,\text{ in probability}.
$$
\end{enumerate}
\end{theorem}

\subsection{Some technical estimates}

For the proof of Theorem~\ref{theo:lin_main} we derive several preliminary results. Note that  (\ref{eq:lin}) is a linear equation and in terms of  $$\cH[l,k]= \prod_{r=l+1}^k (\1+\gamma_r H)\qquad (l,k\in\N,l\leq k)  $$
we get for $n,n_0\in\N$ with $n\geq n_0$
$$
\theta_n = \cH[n_0,n] \theta_{n_0} + \sum_{k=n_0+1}^n \gamma_k \cH[k,n] \Upsilon_k
$$
and
\begin{align}\begin{split}\label{eq_rep}
\bar \theta_n&= \frac 1{\bar b_n} \sum_{k=0}^n b_k \theta_k=\frac 1{\bar b_n} \sum_{k=0}^{n_0-1}  b_k \theta_k+  \frac 1{\bar b_n}  \sum_{k=n_0}^n b_k \Bigl( 
\cH[n_0,k] \theta_{n_0} + \sum_{l=n_0+1}^k \gamma_l \cH[l,k] \Upsilon_l\Bigr)\\
&=\frac 1{\bar b_n} \sum_{k=0}^{n_0-1}  b_k \theta_k+\frac 1{\bar b_n}  \sum_{k=n_0}^n b_k 
\cH[n_0,k] \theta_{n_0} + \frac 1{\bar b_n}
  \sum_{l=n_0+1}^n b_l \bar \cH[l,n] \Upsilon_l,
\end{split}\end{align} 
where
$$
\bar \cH[l,n]= \sum_{k=l}^n \frac{\gamma_l  b_k}{b_l}  \cH[l,k].
$$
In the proof of Theorem~\ref{theo:lin_main} we will show that the first two terms in representation~(\ref{eq_rep}) are asymptotically negligible. As we will show in this section  $\bar\cH[l,n]$ is for large $l$ and $n$ typically close to $-H^{-1}$.

We let for $n\in\N_0$
$$
t_n=\sum_{k=1}^n \gamma_k
$$
and for $l,k\in\N$ with  $l\leq k$ and $s\in[t_{k-1}-t_{l-1},t_k-t_{l-1})$ we set
$$
K_l(s):=k.
$$
Note that for each $l\in\N$, $K_l$ defines a function mapping $[0,\infty)$ onto $\{l,l+1,\dots\}$. Furthermore, we let
$$
F_l(s):= \frac {\gamma_l\,b_{K_l(s)}}{\gamma_{K_l(s)}\, b_l}
$$
Understanding the behaviour of $F_l$ will be the key tool for proving that $\bar\cH[l,n]$ is close to $-H^{-1}$.

\begin{lemma}\label{le:B3'} Assumption (L3) implies that 
\begin{enumerate}
\item[(L3')] $F_l$ converges pointwise to $1$ and
\item[(L3'')] there exists a measurable function $\bar F$ such that $F_l\leq \bar F$ for all $l\geq n_0$  with $n_0\in\N$ sufficiently large, and 
$$
\int_0^\infty \bar F(s) \,(s\vee 1) e^{-Ls}\, ds <\infty.
$$
\end{enumerate}
\end{lemma}

\begin{proof}
Let $\delta>0$. By assumption (L3)  there exists $n_\delta$ such that for all $n\geq n_\delta$
$$
\Big|\frac {b_{n+1}/b_n}{\gamma _{n+1}/\gamma_n}-1\Big| \leq \delta \gamma_n
$$
Now consider $l\geq n_\delta$. The fuction  $F_l$ is constant on each interval $[t_{k-1}-t_{l-1},t_k-t_{l-1})$ and attains the value~$a_k=\frac{\gamma_lb_k}{\gamma_k b_l}$ ($k=l,l+1,\dots$). Note that  $a_l=1$ and  $F_l(s)\leq e^{\delta s}$ for $s\in [0,t_l-t_{l-1})$. By induction we prove that 
$$
F_l(s)\leq e^{\delta s}
$$
for all $s\in [t_{k-1}-t_{l-1},t_k-t_{l-1})$ with $k\in \N_0$. 
Indeed, using the induction hypothesis  we conclude that
$$
a_{k+1}= a_k \frac {b_{k+1}/b_k}{\gamma_{k+1}/\gamma_k}\leq (1+\delta\gamma_k) a_k\leq e^{\delta\gamma_k} F_l(t_{k-1}-t_{l-1})\leq e^{\delta(t_{k-1}-t_{l-1}+\gamma_k)}\leq e^{\delta(t_{k}-t_{l-1})}
$$
and hence, for $s\in [t_{k}-t_{l-1},t_{k+1}-t_{l-1})$, $F_l(s)=a_{k+1}\leq e^{\delta s}$. 
In particular, we get for the choice $n_0\geq n_{L/2}$ that  $\bar F(s)=e^{Ls/2}$ is an integrable majorant in the sense of (L3'').

To prove pointwise convergence (L3') it remains to provide an estimate in the converse direction.
Based on the estimate $e^{-2x}\leq 1-x$ for $x\in[0,\frac 12]$ one argues in complete analogy to before that for  $\delta\in(0,(2\gamma_1)^{-1})$ and   $l\geq n_\delta$ 
$$
F_l(s)\geq e^{-2\delta (s-\gamma_l)}
$$
for all $s\geq 0$.
Since $\delta$ can be chosen arbitrarily small we obtain with the respective upper bound the pointwise convergence of (L3'). 
\end{proof}

\begin{lemma}\label{le:count_set}
Assumption (L2) implies that  for every $C>0$
\begin{enumerate}\item[(L2')]$\displaystyle{
\lim_{n\to\infty} \frac 1n \#\{l\in\{1,\dots,n\}: t_n-t_l\leq C\}=0.}$
\end{enumerate}
\end{lemma}

\begin{proof}Let $\delta>0$ and fix $n_\delta\in\N$ such that $\delta \gamma_{n+1}\geq  n^{-1}$ for all $n\geq n_\delta$. We show that for $l\geq n_\delta$ one has
$$
K_{l}(s)\leq l e^{\delta s}.
$$
Clearly this is true for  $s\in[0,t_{l}-t_{l-1})$. We proceed by induction. Suppose we verified the statement on $[0,t_k-t_{l-1})$ we show that it holds for  $s\in [t_{k}-t_{l-1},t_{k+1}-t_{l-1})$. Using the induction hypothesis we get that 
$$
K_l(s)= k+1= \frac {k+1}k K_l(s-\gamma_{k+1}) \leq e^{1/k} l e^{\delta (s-\gamma_{k+1})}\leq l\,e^{\delta s}.
$$

Note that for $n\geq l\geq n_\delta$
$$
n+1=K_{l+1}(t_{n}-t_{l})\leq (l+1) e^{\delta (t_{n}-t_{l})}
$$
so that $t_{n}-t_{l}\geq \delta^{-1} \log \frac{n+1}{l+1}$. Consequently,
$$
\limsup_{n\to\infty}  \frac 1n \#\{l\in\{1,\dots,n\}: t_n-t_l\leq C\}\leq \limsup_{n\to\infty}  \frac 1n \#\Bigl\{l\in\{1,\dots,n\}:\log \frac{n+1}{l+1}\leq \delta C\Bigr\}=1-e^{-\delta C}
$$
and the statement follows since $\delta>0$ can be chosen arbitrarily small.
\end{proof}

%

\begin{lemma}For a $d\times d$-matrix $A$ it holds \begin{itemize}
\item[(i)] $\|e^A-\1\|\leq e^{\|A\|} \|A\|$
\item[(ii)]  $\|e^A- (\1+A)\|\leq \frac12 e^{\|A\|} \|A\|^2$
\end{itemize}
\end{lemma}

\begin{proof}
One has $e^A-(\1+A)=\sum_{k=2}^\infty \frac 1{k!} A^k$ so that
$$
\|e^A-(\1+A)\|\leq \sum_{k=2}^\infty \frac 1{k!} \|A\|^k\leq \frac 12\|A\|^2 \sum_{k=2}^\infty \frac 1{k!} \|A\|^k  = 
\frac 12 e^\xi \|A\|^2
$$
and (ii) follows. (i) follows by exactly te same argument. 
\end{proof}

\begin{prop}\label{prop345}  Suppose that (L1) holds for $L>0$ and $\eps_0>0$ and let $n_0\in\N$ with $\gamma_{n_0}\leq \eps_0$.
Then for all $m\geq r\geq n_0$
$$
\Bigl\|e^{(t_m-t_r)H} -\prod_{l=r+1}^m (\1+\gamma_l H)\Bigr\|\leq  \|H\|^2 e^{\gamma_{1}(L+\|H\|)} e^{-(t_{m}-t_r)L} \sum_{q=r+1}^m \gamma_q^2 ,
$$
where
$$
t_r=\sum_{l=1}^r \gamma_l.
$$
\end{prop}

\begin{proof}
We write 
\begin{align*}
e^{(t_m-t_r)H} -\prod_{l=r+1}^m (\1+\gamma_l H)= \sum_{q=r+1}^m e^{(t_{q-1}-t_r)H} (e^{\gamma_qH}-(\1+\gamma_q H)) \prod_{l=q+1}^m (\1+\gamma_l H)
\end{align*}
and consider the three terms on the right hand side separately. 

First observe that for $m\geq q\geq n_0$
$$
\Bigl\|\prod_{l=q+1}^m (\1+\gamma_l H)\Bigr\|\leq \prod_{l=q+1}^m \underbrace{(\1-\gamma_l L)}_{\leq e^{-\gamma_l L}}\leq e^{-L(t_m-t_q)}.
$$
Using that
$
e^{(t_{q-1} -t_r)H}= \lim_{n\to\infty} (\1+\frac{t_{q-1}-t_r}n H)^n
$
 the above computation also implies  that $$\|e^{(t_{q-1}-t_r)H}\|\leq e^{-L(t_{q-1}-t_r)}.$$
Consequently,
$$
\Bigl\|e^{(t_m-t_r)H} -\prod_{l=r+1}^m (\1+\gamma_l H)\Bigr\|\leq  \frac 12 e^{-(t_{m}-t_r)L+\gamma_1L} e^{\gamma_1\|H\|} \|H\|^2 \sum_{q=r+1}^m \gamma_q^2 .
$$
\end{proof}

\begin{lemma}\label{le:H-1} Suppose that properties (L1) and (L3) hold. One has 
$$
\bar\cH[l,n]=\frac{\gamma_l}{b_l}\sum_{k=l}^n b_k \cH[l,k] \to -H^{-1}
$$
for $l,n\to\infty$ with $t_n-t_l\to\infty$. Further there exists $n_0\in\N$ such that the operator on the left hand side is uniformly bounded for all $n_0\leq l\leq n$.
\end{lemma}

\begin{proof}Recall that assumption (L3) implies (L3') and (L3''), see Lemma~\ref{le:B3'}. By (L1) we can fix $L>0$ and $n_0\in\N$ such that for $n\geq n_0$, $\|\1+\gamma_n H\|\leq 1-L\gamma_n$. For $n\geq l\geq n_0$ we consider
 $I_1=I_1(l,n)=\bar \cH[l,n]$, 
$$
I_2= I_2(l,n)=\frac{\gamma_l}{b_l}\sum_{k=l}^n b_k e^{(t_k-t_l)H} \text{ \ and \ } I_3=I_3(l,n)=\sum_{k=l}^n \gamma_k e^{(t_k-t_l)H}$$
and omit the $(l,n)$-dependence in the notation.

We analyse $\|I_1-I_2\|$. One has 
$$
I_1-I_2= \frac{\gamma_l}{b_l}\sum_{k=l}^n b_k (\cH[l,k]-e^{(t_k-t_l)H})
$$
By Proposition~\ref{prop345}, 
\begin{align*}
\|\cH[l,k]-e^{(t_k-t_l)H}\|&\leq  \|H\|^2 e^{\gamma_{1}(L+\|H\|)} e^{-(t_{k}-t_l)L}  \sum_{q=l+1}^k \gamma_q^2\\
&\leq  \|H\|^2e^{\gamma_{1}(L+\|H\|)}  \gamma_{l}  e^{-(t_{k}-t_l)L} (t_k-t_l) 
\end{align*}
so that
\begin{align*}
\|I_1-I_2\| &\leq \underbrace{ \|H\|^2e^{\gamma_{1}(L+\|H\|)} }_{=:C} \gamma_{l}\sum_{k=l}^n  \frac{\gamma_lb_k}{b_l\gamma_k} e^{-(t_{k}-t_l)L} (t_k-t_l) \gamma_k\\
&=C \gamma_l \int_0^{t_n-t_{l-1}} F_l(s) e^{-(t_{K_l(s)}-t_l)L}(t_{K_l(s)}-t_l)\, ds\\
&\leq C e^{\gamma_1 L} \gamma_l \int_0^{t_n-t_{l-1}} F_l(s) e^{-sL}s \, ds ,
\end{align*}
where we used that $t_{K_l(s)}-t_l\leq t_{K_l(s)-1}-t_{l-1}\leq s < t_{K_l(s)}-t_{l-1}= t_{K_l(s)}-t_l+\gamma_l$ in the previous step. Property (L3'') implies that there exists an integrable majorant for the latter integrand. Hence $\|I_1-I_2\|$  is uniformly bounded and converges to zero as $l,n\to\infty$ with $l\leq n$.

We analyse $\|I_2-I_3\|$. One has
\begin{align*}
I_2-I_3= \sum_{k=l}^n \Bigl(\frac{\gamma_l b_k}{b_l \gamma_k}-1\Bigr) \gamma_k e^{(t_k-t_l) H}= \int_0^{t_n-t_{l-1}} \Bigl( \frac {\gamma_l b_{K_l(s)}}{b_l \gamma_{K_l(s)}}-1\Bigr) e^{(t_{K_l(s)}-t_l)H} \, ds
\end{align*}
and using that  $t_{K_l(s)}-t_l\leq t_{K_l(s)-1}-t_{l-1}\leq s$ and the definition of $F_l$ we get that
$$
\|I_2-I_3\|\leq \int_0^{t_n-t_{l-1}} |F_l(s)-1|\, e^{-Ls} \, ds.
$$
As before there exists an integrable majorant (L3''). Hence $\|I_2-I_3\|$ is uniformly bounded and with dominated convergence and (L3') we conclude that the latter integral converges to zero as $l,n\to\infty$ with $l\leq n$.

We analyse $\|I_3+H^{-1}\|$. One has 
\begin{align*}
I_3+H^{-1} =\int_0^{t_n-t_{l-1}} e^{(t_{K_l(s)}-t_l)H} \, ds - \int_0^\infty e^{s H}\,ds
\end{align*}
and noting that $t_{K_l(s)}-t_l\leq t_{K_l(s)-1}-t_{l-1}\leq s < t_{K_l(s)}-t_{l-1}= t_{K_l(s)}-t_l+\gamma_l$ we conclude that
\begin{align*}
\|I_3-H^{-1} \|& \leq \int_0^{t_n-t_{l-1}} \|\underbrace{e^{(t_{K_l(s)}-t_l)H}-e^{sH}}_{=e^{(t_{K_l(s)}-t_l)H}(\1-e^{s-(t_{K_l(s)}-t_l)})}\| \, ds + \int_{t_n-t_{l-1}}^\infty \underbrace{\|e^{s H}\|}_{\leq e^{-sL}}\,ds\\
&\leq  \int_0^{t_n-t_{l-1}} \underbrace{e^{-L(t_{K_l(s)}-t_l)}}_{\leq e^{-Ls +L\gamma_l}}   \gamma_l\|H\| e^{\gamma_l \|H\|}\, ds + L^{-1} e^{-L(t_n-t_{l-1})}
\end{align*}
which converges to zero as $l,n\to\infty$ with $t_n-t_l\to\infty$. Here we used that $\|e^A-\1\|\leq \|A\| e^{\|A\|}$ for  a matrix $A$. Further $\|I_3\|$ is uniformly bounded since
$$
\|I_3\|\leq \int_0^{t_n-t_{l-1}} \| e^{(t_{K_l(s)}-t_l)H}\|\, ds \leq \int_0^\infty e^{-L(s-\gamma_l)}\, ds\leq \frac{e^{L\gamma_1}}L.
$$
\end{proof}

We state a classical  result of linear algebra that is usually used in the theoretical analysis of   stochastic approximation. 

\begin{lemma}\label{le:norm_ex}
If for $L>0$, $H\in\R^{d\times d}$ is an $L$-contracting matrix in the sense that
$$
\sup\{\mathrm{Re}(\lambda): \lambda\text{ e.v.\ of }H\}<-L,
$$
then there exists a norm $\|\cdot\|$ on $\R^{d\times d}$ induced by a scalar product on $\R^d$ and $\eps_0>0$ such that for all $\eps\in[0,\eps_0]$
$$
\|\1+\eps H\|\leq 1-\eps L.
$$
\end{lemma}

\subsection{Proof of Theorem~\ref{theo:lin_main}}

We note that by Lemmas~\ref{le:B3'} and \ref{le:count_set} properties (L2'), (L3') and (L3'') are satisfied and we choose $n_0\in\N$ sufficiently large such that all properties hold. The proof of Theorem~\ref{theo:lin_main} relies on an asymptotic analysis of 
\begin{align*}
\Xi_n:  =\ \frac 1{\bar b_n}
  \sum_{l=n_0+1}^n b_l \bar \cH[l,n] \Upsilon_l
\end{align*}
for $n\geq n_0$. Note that by Lemma~\ref{le:H-1} all operators 
$\bar \cH[l,n]$ appearing above are uniformly bounded and furthermore $\bar \cH[l,n]\to -H^{-1}$ as $l,n\to\infty$ with $t_n-t_l\to\infty$. For each of the three statements of Theorem~\ref{theo:lin_main} we carry out the analysis separately. 
In each case it will be easy to see that the remaining terms of the representation~(\ref{eq_rep}) are asymptotically negligible.

\begin{proof}[Proof of Theorem~\ref{theo:lin_main} I]
For every $\eps>0$ there exists $l_0\in\N$ and $C_1>0$ such that
$$
\|\bar\cH[l,n]+H^{-1}\|\leq \eps
$$
for all $n\geq l> l_0\geq n_0$ with $t_n-t_l\geq C_1$. Denote by $L(n)$ the maximal index $l$ with $t_n-t_l\geq C$ and note that $\lim_{n\to\infty} (n-L(n))/n=0$ by  Lemma~\ref{le:count_set}. Using the uniform boundedness of $\bar\cH[l,n]$ we get that for an appropriate constant $C_2$, for sufficiently large $n$ ($L(n)\geq l_0$)
\begin{align}
\Bigl\|\Xi_n+ &\  \frac 1{\bar b_n}
  \sum_{l=n_0+1}^n b_l H^{-1} \Upsilon_l \Bigr\|\leq \frac 1{\bar b_n}
  \sum_{l=n_0+1}^n b_l \| \bar \cH[l,n]+H^{-1}\|\| \Upsilon_l\|
  \nonumber\\
 \label{eq972435671} \leq  &\ C_2 \frac 1{\bar b_n}
  \sum_{l=n_0+1}^{l_0} b_l  \|\Upsilon_l\| + \eps \frac 1{\bar b_n}
  \sum_{l=l_0+1}^{L(n)} b_l  \|\Upsilon_l\| +C_2 \frac 1{\bar b_n}
  \sum_{l=L(n)+1}^{n} b_l  \|\Upsilon_l\|
\end{align}
By assumption,  one has $\lim_{n\to\infty} \delta_n^{-1}\Upsilon_n=\mu$, on $\{\theta_n\to0\}$. Recalling~(\ref{eq7845}) we conclude that the first and third term in~(\ref{eq972435671}) are of order $o(\bar b_n^{-1}\sum_{k=1}^n b_k\delta_k)$ on $\{\theta_n\to0\}$ and the second term is for sufficiently large $n$ smaller than $2\eps \bar b_n^{-1}  \|\mu\| \sum_{k=1}^n b_k \delta_k$. Since $\eps>0$ is arbitrary we get that on $\{\theta_n\to0\}$
$$
\Bigl\|\Xi_n+   \frac 1{\bar b_n}
  \sum_{l=n_0+1}^n b_l H^{-1}\Upsilon_l \Bigr\|= o\Bigl(\frac1{\bar b_n} \sum_{k=1}^n b_k \delta_k\Bigr).
$$
On  $\{\theta_n\to0\}$ we have that $\delta_l^{-1}\Upsilon_l\to \mu$ so that with the same argument as above
$$
\Bigl\|\frac 1{\bar b_n}
  \sum_{l=n_0+1}^n b_l H^{-1}(\Upsilon_l-\delta_l \mu) \Bigr\|= o\Bigl(\frac1{\bar b_n} \sum_{k=1}^n b_k \delta_k\Bigr).
$$
Consequently, on $\{\theta_n\to0\}$
$$
\lim_{n\to\infty} \Bigl(\frac 1{\bar b_n}
  \sum_{l=1}^n b_l \delta_l\Bigr)^{-1} \Xi_n=-H^{-1} \mu.
$$
In view of~(\ref{eq_rep}) we have
$$
\bar \theta_n=\frac 1{\bar b_n} \sum_{k=0}^{n_0-1}  b_k \theta_k+\frac 1{\bar b_n}  \sum_{k=n_0}^n b_k 
\cH[n_0,k] \theta_{n_0} + \Xi_n
$$
and the first and second term on the right hand side are of lower order. Indeed, the terms are of order $\mathcal O((\bar b_n)^{-1})$ (Lemma~\ref{le:H-1}) which is asymptotically negligible.
\end{proof}

\begin{proof}[Proof of Theorem~\ref{theo:lin_main} II]
We verify the validity of a central limit theorem (see \cite[Cor. 3.1]{HaHe80}) for 
\begin{align*}
\Xi_n= \frac 1{\bar b_n}
  \sum_{l=n_0+1}^n b_l \bar \cH[l,n] \Upsilon_l.
\end{align*}
We note that the summands are martingale differences and consider
\begin{align*}
\frac 1{\bar b_n^2}
  \sum_{l=n_0+1}^n b_l^2 \cov( \bar \cH[l,n] \Upsilon_l|\cF_{l-1})=\frac 1{\bar b_n^2}
  \sum_{l=n_0+1}^n b_l^2  \bar \cH[l,n] \cov( \Upsilon_l|\cF_{l-1}) \bar \cH[l,n]^\dagger
\end{align*}
First we show convergence of the conditional covariances on $\{\theta_n\to0\}$.
One has
\begin{align*}
\| \bar \cH[l,n] &\cov( \Upsilon_l|\cF_{l-1}) \bar \cH[l,n]^\dagger- \delta_l^2  H^{-1}\Gamma (H^{-1})^\dagger \|\\
 &\leq \|\bar \cH[l,n]+H^{-1}\| \| \cov( \Upsilon_l|\cF_{l-1})\| \| \bar \cH[l,n]^\dagger\| \\
&+ \|H^{-1}\| \|\cov( \Upsilon_l|\cF_{l-1}) -\delta^2_l \Gamma\| \|\bar \cH[l,n]^\dagger\|+ \delta^2_l \| H^{-1}\| \|\Gamma\|\bar \cH[l,n]^\dagger+ (H^{-1})^\dagger\|.
\end{align*}
As a consequence of  the convergence of the conditional covariances one has   on $\{\theta_n\to0\}$  for sufficiently large~$l$, say $l\geq L$ with $L$ being an appropriate random variable,
$$
 \| \cov( \Upsilon_l|\cF_{l-1})\| \leq C_1 \delta_l^2.
$$
Further by Lemma~\ref{le:H-1} there exists a constant $C_2$ such that $\|\cH[l,n]\|\vee \|\cH[l,n]^\dagger\|\leq C_2$ for all $n_0\leq l\leq n$.
%
As in the proof of part I one shows that on $\{\theta_n\to0\}$ up to nullsets
\begin{align*}
\frac 1{\bar b_n^2}&
  \sum_{l=n_0+1}^n b_l^2  \|\bar \cH[l,n]+H^{-1}\| \| \cov( \Upsilon_l|\cF_{l-1})\| \| \bar \cH[l,n]^\dagger\|\\
  & \leq \frac 1{\bar b_n^2}
  \sum_{l=n_0+1}^{L\wedge n} b_l^2 \|\bar \cH[l,n]+H^{-1}\| \| \cov( \Upsilon_l|\cF_{l-1})\| \| \bar \cH[l,n]^\dagger\| + C_1 C_2  \frac 1{\bar b_n^2}
  \sum_{l=L\wedge n+1}^{n} (b_l \delta_l)^2   \|\bar \cH[l,n]+H^{-1}\|
\end{align*}
is of order $o( \sigma_n^2)$ with $\sigma_n^2=\bar b_n^{-2}   \sum_{l=1}^{n} (b_l \delta_l)^2$. The same is true for the terms
$$
\frac 1{\bar b_n^2}
  \sum_{l=n_0+1}^n b_l^2 \|H^{-1}\| \|\cov( \Upsilon_l|\cF_{l-1}) -\delta^2_l \Gamma\| \|\bar \cH[l,n]^\dagger\|
  $$
  and
$$
\frac 1{\bar b_n^2}
  \sum_{l=n_0+1}^n (b_l \delta_l)^2 \| H^{-1}\| \|\Gamma\|    \|\bar \cH[l,n]^\dagger+ (H^{-1})^\dagger\|.
$$
Consequently, on $\{\theta_n\to0\}$
$$
  \sigma_n^{-2}
   \frac 1{\bar b_n^2}
  \sum_{l=n_0+1}^n b_l^2 \cov( \bar \cH[l,n] \Upsilon_l|\cF_{l-1}) \to H^{-1}\Gamma (H^{-1})^\dagger
$$
and the conditional covariances converge.
It remains to verify the conditional Lindeberg condition. 
Recall that
there exists $C_2$ such that  $\|\bar \cH[l,n]\|\leq C_2$ for $n_0\leq l\leq n$. 
Consequently, for $\eps>0$,
$$
 \sigma_n^{-2} \sum_{l=n_0+1}^n \E\Bigl[\1_{\{\|\frac {b_l}{\bar b_n}  \bar \cH[l,n] \Upsilon_l\| / \sigma_n\geq \eps\}} \Bigl\|\frac {b_l}{\bar b_n}  \bar \cH[l,n] \Upsilon_l\Bigr\|^2\Big|\cF_{n-1}\Bigr]\leq C_2^2 \sigma_n^2 \sum_{l=n_0+1}^n \frac{b_l^2} {\bar b_n^{2}} \E\Bigl[\1_{\{ \| \Upsilon_l\| \geq \frac {\eps\bar b_n \sigma_n }{C_2 b_l}\}}
  \| \Upsilon_l\|^2\Big|\cF_{l-1}\Bigr]
$$
and the latter term tends to zero in probability on $\{\theta_n\to0\}$ by assumption.
Thus the central limit theorem is applicable and on $\{\theta_n\to 0\}$
$$
\frac1{\sigma_n} \Xi_n\Rightarrow \mathcal N(0, H^{-1}\Gamma (H^{-1})^\dagger).
$$
In view of (\ref{eq_rep}) the statement of the proposition follows once we showed that on $\{\theta_n\to0\}$ the terms
$$
\frac 1{\bar b_n} b_k \theta_k ,\text { \ for $k=0,\dots,n_0-1$, \ \ and \ \ } \frac 1{\bar b_n}\sum_{k=n_0}^n b_k \cH[n_0,k] \theta_{n_0}
$$
are of order $o(\sigma_n)$. The terms are of order $\mathcal O(\bar b_n^{-1})$ (since $\cH[n_0,k]$ is uniformly bounded by Lemma~\ref{le:H-1}) which is of order  $o(\sigma_n)$ since $\sum_{k=1}^\infty (b_k\delta_k)^2=\infty$.
\end{proof}

\begin{proof}[Proof of Theorem~\ref{theo:lin_main} III]
We control  $\E[\| \Xi_n\|]$. By Lemma~\ref{le:H-1} there exists a constant $C_1$ such that for sufficiently large $n_0$ and all $n\geq n_0$
\begin{align*}
\E[\|\Xi_n\|]\leq  \frac 1{\bar b_n} \E\bigl[   \bigl\|  \sum_{l=n_0+1}^n b_l \bar \cH[l,n] \Upsilon_l\bigr\|\bigr] \leq \frac 1{\bar b_n} C_1   \sum_{l=n_0+1}^n b_l\,  \E[\|\Upsilon_l\|]= o(\eps_n)
\end{align*}
so that $\Xi_n/\eps_n\to 0$, in probability. Moreover, by assumption $\lim_{n\to\infty} \eps_n^{-1} \bar b_n=0$ which implies that also the
remaining terms in the representation (\ref{eq_rep}) are of order $o(\eps_n)$.
\end{proof}

\section{Proof of Theorem~\ref{theo_RP_main}}\label{proof_main_SA}

In this section, we prove the  central limit theorem for Ruppert-Polyak averaging. We assume the setting of  Section~\ref{sec_1_1} and in particular we assume that all assumptions imposed in Theorem~\ref{theo_RP_main} are satisfied.\smallskip

Since all assumptions are translation invariant it suffices to show that the statement is true in the case where $\theta^*=0$. Note that the linear equation~(\ref{eq:lin}) is satisfied for the choice
$$
\Upsilon_n=\frac 1{\gamma_n}(\theta_n- \theta_{n-1})-H\theta_{n-1}.
$$
Moreover the case that $\theta_{n-1}+ \gamma_n (f(\theta_{n-1}) + D_n+ R_n)$  is  in $D$ enters on $\{\theta\to0\}$ for all but finitey many $n$'s (since $0$ is not allowed to be a boundary point) and for all these $n$'s one has
$$
\Upsilon_n= f(\theta_{n-1})-H\theta_{n-1} + D_n+ R_n.
$$
We will represent $\Upsilon_n$ as the sum of three terms and control their contribution by Theorem~\ref{theo:lin_main}. According to Lemma~\ref{le:norm_ex} we fix an inner product $\langle \cdot,\cdot\rangle$  for which assumption (L1) is satisfied and denote by $\|\cdot\|$ the respective norm. Assumptions (L2) and (L3) are satisfied by the choice of parameters.

Choose  $C_0> \mathrm{trace}(\Gamma)$ (with the trace being taken in the Hilbert space $\langle\cdot,\cdot\rangle$) and $\eps'>\eps>0$  such that Assumptions (D) and (E) are satisfied and  such that~(\ref{eq893467}) holds as a consequence of Theorem~\ref{thm1}. We consider
\begin{align*}
\Upsilon^\mathrm{(diff)} _n &=D_n \1\{\theta_{n-1}\in B(0,\eps')\} \\
\Upsilon^\mathrm{(rem)}_n& = (f(\theta_{n-1})-H\theta_{n-1}) \1 \{\theta_{n-1} \in B(0,\eps) \}\\
\Upsilon^\mathrm{(bias)}_n & = \Upsilon_n-\Upsilon_n^\mathrm{(diff)}-\Upsilon_n^\mathrm{(rem)}
\end{align*}
By linearity one has
$$
\theta_n=\theta_n^\mathrm{(bias)}+\theta_n^\mathrm{(diff)}+\theta_n^\mathrm{(rem)} \text{ \ and \ } \bar\theta_n=\bar\theta_n^\mathrm{(bias)}+\bar\theta_n^\mathrm{(diff)}+\bar\theta_n^\mathrm{(rem)}
$$
where the corresponding processes are the solutions of (\ref{eq:lin}) when choosing the corresponding~$\Upsilon$ and starting the first process in $\theta_0$ and the second and third one in zero. 

1) On the set $\{\theta_n\to 0\}$ the projection $\Pi$ will take effect only in a finite number of steps.
 Furthermore, the indicators in the definitions of $\Upsilon^{\mathrm{(diff)}} _n$ and $\Upsilon^\mathrm{(rem)} _n $ will be one for all but finitely many $n$. Hence, on $\{\theta_n\to0\}$ one has for all but finitely many $n$  that $\Upsilon^\mathrm{(bias)}_n=R_n$ so that by part I of Theorem~\ref{theo:lin_main}
$$
\lim_{n\to\infty}(\eps_n^\mathrm{bias})^{-1} \bar \theta_n^\mathrm{(bias)}= H^{-1} \mu,\text{ \ up to nullsets.}
$$

2) By definition, the process  $(\Upsilon_n^\mathrm{diff})_{n\in\N}$ is a sequence of square integrable martingale differences and on  $\{\theta_n\to 0\}$, one has  $\Upsilon_n^\mathrm{(diff)}=D_n$ for all but finitely many $n\in\N$. Hence, it follows by part II of Theorem~\ref{theo:lin_main} that on $\{\theta_n\to 0\}$
$$
(\eps_n^\mathrm{diff} )^{-1} 
 \bar \theta_n^\mathrm{(diff)} \Rightarrow \mathcal N(0,  H^{-1} \Gamma(H^{-1})^\dagger).
$$

3) For $n_0\in\N$ let 
$$
\Upsilon^{\mathrm{(rem}, n_0)}_{n} = (f(\theta_{n-1})-H\theta_{n-1}) \1 \{\theta_{l} \in B(0,\eps)\text{ for all }l=n_0,\dots,n-1 \}
$$
Using property (B) together with Theorem~\ref{thm1} we conclude that there are  constants $C_1$ and $C_2$ not depending on the choice of $n_0$ such that for sufficiently large $n\in\N$
 one has
\begin{align*}
 \E[\|\Upsilon^{\mathrm{(rem}, n_0)}_{n}\|] &\leq C_1 \E[\|\theta_{n-1}\|^{1+\lambda}  \1 \{\theta_{l} \in B(0,\eps')\text{ for }l=n_0,\dots,n-1 \}]\\
& \leq C_1  \E[\|\theta_{n-1}\|^2  \1 \{\theta_{l} \in B(0,\eps')\text{ for }l=n_0,\dots,n-1 \}]^{(1+\lambda)/2}
\leq C_2 \delta_{n-1}^{1+\lambda}.
\end{align*}
Recall that assumption (A) implies that $\delta_{n-1}=\mathcal O(\delta_n)$ so that there exists a constant $C_3$ not depending on $n_0$ such that $ \E[\|\Upsilon^{\mathrm{(rem}, n_0)}_{n}\|]\leq C_3 \delta_{n}^{1+\lambda}$ for sufficiently large $n$.

First suppose that~(\ref{eq74598}) is of order $o(\eps_n^\mathrm{bias})$.
Part III of Theorem~\ref{theo:lin_main} is applicable and we obtain that
\begin{align}\label{eq:84578}
\lim_{n\to\infty}  (\eps_n^\mathrm{bias})^{-1}  \bar\theta^{\mathrm{(rem,} n_0)}_n=0,\text{ in probability}.
\end{align}
Note that on  $$\Omega_{n_0}=\{\theta_n\to0\}\cap\{ \theta_{l} \in B(0,\eps')\text{ for all }l=n_0,\dots\}$$
 $\Upsilon^{\mathrm{(rem}, n_0)}$ agrees with $\Upsilon^\mathrm{(rem)}$. Since $\bigcup_{n_0\in\N} \Omega_{n_0}=\{\theta_n\to\theta^*\}$, we conclude that (\ref{eq:84578}) still remains true on $\{\theta_n\to \theta^*\}$ when replacing $ \bar\theta^{\mathrm{(rem,} n_0)}_n$ by $\bar\theta^{\mathrm{(rem)}}_n$. We choose $\vartheta_n:=\bar\theta_n^\mathrm{bias}+\bar \theta_n^\mathrm{(rem)}$ and observe that in combination with 1) and 2) we have proved the statement.
 
 Conversely in the case where~(\ref{eq74598}) is of order $o(\eps_n^\mathrm{diff})$ one concludes similarly with 1) and 2) that the statement follows for $\vartheta_n=\bar\theta_n^\mathrm{(bias)}$.

\section{Proof of Theorem~\ref{theo:main_pol}}\label{sec4}

Note that the states $(\theta_n)_{n\in\N_0}$ of the algorithm satisfy equation~(\ref{dynsys2}) when choosing
$$
R_n=\begin{cases} \E[Z(\mathbb U_n;\theta_{n-1},s_n,K_n)|\cF_{n-1}]-f(\theta_{n-1}), & \text{ if }\theta_{n-1}\in D_{\theta^*},\\
Z(\mathbb U_n;\theta_{n-1},s_n,K_n)-f(\theta_{n-1}), &\text{ else},
\end{cases}
$$
and 
$$
D_n= \begin{cases} Z(\mathbb U_n;\theta_{n-1},s_n,K_n)- \E[Z(\mathbb U_n;\theta_{n-1},s_n,K_n)|\cF_{n-1}],&\text{ if }\theta_{n-1}\in D_{\theta^*},\\
0, &\text{ else},
\end{cases}
$$
where $D_{\theta^*}$ is the neighbourhood of $\theta^*$ appearing in Assumptions~(\ref{as_tail_1}) and (\ref{as_cost2_2}). We set  
$$  \delta_n^\mathrm{bias}=M^{-\alpha s_n}, \ (\delta_n^\mathrm{diff})^2= M^{(1-\beta)s_n}/n^\varphi \ \text{ and } \ \delta_n=n^{-(\varphi+1)\mathfrak r+\frac 12(1-\psi)}$$
and
$$
\epsilon_n^\mathrm{bias}=\bar b_n^{-1} \sum_{l=1}^n b_l \delta_l^\mathrm{bias}\text{ \ and \ }\epsilon_n^\mathrm{diff}=\bar b_n^{-1}\sqrt{ \sum_{l=1}^n (b_l \delta_l^\mathrm{diff})^2}.
$$

1) \emph{Verification of the assumptions of Theorem~\ref{theo_RP_main}.} 
Property (A) is a direct consequence of the choice of the parameters and (B) holds by assumption. 

To verify property (C) we first derive the asymptotics of $M^{s_n}$. Using the definition of $\xi_n$ and the asymptotic equivalence  $\bar K_n\sim \kappa_K n^{\varphi+1}$ we conclude that for sufficiently large $n$
\begin{align}\label{eq:M_sn}
M^{s_n}=\kappa_s \bar K_n^{\frac 1{2\alpha-\beta+1}} M^{-\xi_n}\sim \kappa_s (\kappa_K n^{\varphi+1})^{\frac 1{2\alpha-\beta+1}} M^{-\xi_n}.
\end{align}
Consequently,
\begin{align}\label{eq234975}
\delta_n^\mathrm{bias}\approx n^{-(\varphi+1)\mathfrak r} \ \text{ and } \ (\delta_n^\mathrm{diff})^2\approx n^{(\varphi+1)\frac{1-\beta}{2\alpha-\beta+1} -\varphi}=n^{-2(\varphi+1)\mathfrak r+1}
\end{align}
so that $b_n\delta_n^\mathrm{bias} \approx n^{\rho-(\varphi+1)\mathfrak r}$  and $(b_n\delta_n^{\mathrm{diff}})^2\approx n^{-1+2(\rho+1)-2(\varphi+1)\mathfrak r}$
with the previous exponents being strictly bigger than $-1$ as a consequence of $\mathfrak r<\frac12$ and the second inequality of~(\ref{eq:as_1}). Now~(C) follows by elementary analysis. We mention that this also implies that
\begin{align}\label{eq256235}
\epsilon_n^\mathrm{bias} \approx  \epsilon_n^\mathrm{diff}\approx n^{-(\varphi+1)\mathfrak r}\gg \bar b_n^{-1}\log n.
\end{align}

Next we verify (D). On $\{\theta_n\to\theta^*\}$, one hsa
$$
(\delta_n^\mathrm{bias})^{-1}R_n= M^{\alpha s_n}  (\E[F_{s_n}(\theta,U)]\big|_{\theta=\theta_{n-1}}-f(\theta_{n-1}))\to \mu
$$
as consequence of~(\ref{as_bias1_1}). 
Note that by~(\ref{eq234975}) and $\psi<1$, $\delta_n^\mathrm{bias}=o(\delta_n)$. Furthermore, for  $\eps'>0$ one has
$$
\E[\1_{\{\theta_{n-1}\in B(\theta^*,\eps')\}} |R_n|^2]^{1/2} \leq \sup_{\theta\in B(\theta^*,\eps')} |\E[F_{s_n}(\theta,U)]-f(\theta)|.
$$
Hence (D) follows if we can show that for some $\eps'>0$
$$
\limsup_{n\to\infty} \ (\delta_n^{\mathrm{bias}})^{-1} \sup_{\theta\in B(\theta^*,\eps'')} |\E[F_{s_n}(\theta,U)]-f(\theta)|<\infty.
$$
If this were not be true for any $\eps'>0$ one could define an increasing $\N$-valued sequence   $(n_m)_{m\in\N}$ together with a $D$-valued sequence $(\zeta_{m})_{m\in\N}$ such that $\zeta_m\to \theta^*$ and $(\delta_{n_m}^{\mathrm{bias}})^{-1}  |\E[F_{s_{n_m}}(\zeta_m,U)]-f(\zeta_m)|\to \infty$. This would contradict property~(\ref{as_bias1_1}).

Next we verify (E). Note that if $\theta_{n-1}$ attains a value in $D_{\theta^*}$, then
\begin{align}\begin{split}\label{eq835734}
 \cov(D_n|\cF_{n-1}) & = \cov(Z(\mathbb U_n; \theta_{n-1},s_n,K_n)|\cF_{n-1})\\
&= \sum_{l=1}^{s_n} \frac 1{N_l(s_n,K_n)} \,\cov(F_l(U,\theta)-F_{l-1}(U,\theta))\big|_{\theta=\theta_{n-1}}.
\end{split}\end{align}
Recall that $N_l(s_n,K_n)$ is formed by rounding $\frac {K_n}{M^{s_n}} M^{\frac{\beta+1}2(s_n-l)}$ and for each $n\in\N$ the smallest value is obtained for $l=s_n$. As consequence of the first assumption in~(\ref{eq:as_1}) one has $\varphi-\frac{\varphi+1}{2\alpha-\beta+1}>0$ so that
$$
\frac {K_n}{M^{s_n}}\approx n^{\varphi-\frac{\varphi+1}{2\alpha-\beta+1}}\to \infty
$$
and rounding has no effect on the asymptotics of the sum appearing in~(\ref{eq835734}).
We thus get that on $\{\theta_n\to\theta^*\}$
\begin{align}\begin{split}\label{eq835734-333}
 \cov(D_n|\cF_{n-1}) & = \sum_{l=1}^{s_n} \frac {1+o(1)}{\frac {K_n}{M^{s_n}} M^{\frac {\beta+1}2(s_n-l)}} \,\cov(F_l(U,\theta)-F_{l-1}(U,\theta))\big|_{\theta=\theta_{n-1}}\\
 & =\frac {M^{\frac {1-\beta}2 s_n}}{K_n} \sum_{l=1}^{s_n} (1+o(1)) M^{\frac {1-\beta}2l} \, M^{\beta l} \,\cov(F_l(U,\theta)-F_{l-1}(U,\theta))\big|_{\theta=\theta_{n-1}}.
 \end{split}\end{align}
By assumption~(\ref{as_var1_1})  one has on  $\{\theta_n\to\theta^*\}$ for $n,l\to \infty$ 
$$M^{\beta l} \,\cov(F_l(U,\theta)-F_{l-1}(U,\theta))\big|_{\theta=\theta_{n-1}}\to \Gamma.$$
Furthermore, (\ref{as_tail_1}) implies existence of a constant $C$ such that for all $\theta\in D_{\theta^*}$ and $l\in\N$
$$\bigl \|M^{\beta l} \,\cov(F_l(U,\theta)-F_{l-1}(U,\theta))\bigr\|\leq C.$$
Consequently, it follows that  on  $\{\theta_n\to\theta^*\}$
 \begin{align*}
\lim_{n\to\infty} &(\delta_n^\mathrm{diff})^{-2} \cov(D_n|\cF_{n-1}) =\lim_{n\to\infty} \frac {M^{\frac {1-\beta}2 s_n}}{(\delta_n^\mathrm{diff})^{2}K_n} \sum_{l=1}^{s_n} M^{\frac {1-\beta}2l}\Gamma=\frac1{\kappa_K}  \frac 1{1- M^{-\frac {1-\beta}2}}\frac1{\varphi+1} \Gamma=: \Gamma'.
\end{align*}

Next we show the second part of (E). Recall that we need to show for arbitrarily fixed $\eps>0$ and $\iota_l^{(n)}=\eps \bar b_n \epsilon_n^\mathrm{diff}/b_l$ ($n,l\in\N$ with $l\leq n$) one has  
\begin{align}\label{eq83537}
\lim_{n\to\infty} (\epsilon_n^\mathrm{diff})^{-2}\sum_{l=1}^n \frac {b_l^2}{\bar b_n^2}\,\E[\1_{\{|D_l|>\iota_l^{(n)}\}}|D_l|^2|\cF_{l-1}]=0, \text{ \ in probability}.
\end{align}
First note that for $q>2$ as in Assumption~(\ref{as_tail_1}) one has
$$
\E[|D_l|^q|\cF_{l-1}]= \1_{B_\eps'(\theta^*)}(\theta_{l-1})\, \E\bigl[|Z(\mathbb U;\theta,s_n,K_n)- \E[Z(\mathbb U;\theta,s_n,K_n)]|^q\bigr]\Big|_{\theta=\theta_{l-1}}.
$$
As consequence of the Burkholder-Davis-Gundy inequality and Assumption~(\ref{as_tail_1}) there exists a constant $C$ only depending on $q$, the dimension $d$ and $C_\mathrm{TAIL}$ such that for all  $\theta\in B_{\eps'}(\theta^*)$ and $n\in\N$ one has 
\begin{align*}
\E\bigl[|Z(\mathbb U;\theta,s_n,K_n)- \E[Z(\mathbb U;\theta,s_n,K_n)]|^q\bigr]^{2/q}\leq C\, \sum_{r=1}^{s_n}\frac 1{N_r(s_n,K_n)} M^{-\beta r}.
\end{align*}
The term on the right hand side does not depend on the choice of $\theta$ and  following the same arguments as in~(\ref{eq835734}) we conclude that  it is of order $\mathcal O((\delta_n^\mathrm{diff})^2)$. 
We note that $\E[\1_{\{|D_l|>\iota_l^{(n)}\}} |D_l|^2|\cF_{l-1}]\leq (\iota_l^{(n)})^{-(q-2)}\,\E[|D_l|^q|\cF_{l-1}]$ so that~(\ref{eq83537}) follows once we showed that
$$
\lim_{n\to\infty} (\epsilon_n^\mathrm{diff})^{-2}\sum_{l=1}^n \frac {b_l^2}{\bar b_n^2} (\iota_l^{(n)})^{-(q-2)}(\delta_l^\mathrm{diff})^q =0 \text{ \ or, equivalently, \ }\lim_{n\to\infty} (\epsilon_n^\mathrm{diff})^{-q}\sum_{l=1}^n \frac {b_l^q}{\bar b_n^q} (\delta_l^\mathrm{diff})^q =0.
$$
Using the asymptotic formulas for $\delta_n^\mathrm{diff}$ and $\epsilon_n^\mathrm{diff}$, see~(\ref{eq234975}) and~(\ref{eq256235}), we conclude that the latter term is of order $n^{-\frac 12(q-2)}$ which proves the second part of (E).

To verify the third part we notice that for $\eps'>0$ with  $B(\theta^*,\eps')\subset D_{\theta^*}$ one has
\begin{align*}
\E[\1_{B(\theta^*,\eps')}(\theta_{n-1})\,|D_n|^2]\leq C_{\mathrm{TAIL}} \sum _{r=1}^{s_n} \frac 1{N_r(s_n,K_n)} M^{-\beta r}\approx (\delta_n^\mathrm{diff})^2
\end{align*}
so that together with~(\ref{eq234975})
\begin{align*}
\sqrt{\gamma_n}\,\E[\1_{B(\theta^*,\eps')}(\theta_{n-1})\,|D_n|^2]^{1/2}=\mathcal O\bigl(n^{-\frac\psi2-(\varphi+1)\mathfrak r+\frac 12}\bigr)=\mathcal O(\delta_n).
\end{align*}

We verify (F). One has
$$
\frac 1{\bar b_n} \sum_{l=1}^n b_l \delta_l^{1+\lambda}=\frac 1{\bar b_n} \sum_{l=1}^n l^{\rho+(1+\lambda)(-(\varphi+1)\mathfrak r-\frac 12(\psi-1))}\approx n^{(1+\lambda)(-(\varphi+1)\mathfrak r-\frac12(\psi-1))},
$$
if the exponent in the sum is strictly bigger than $-1$. In that case $\frac 1{\bar b_n} \sum_{l=1}^n b_l \delta_l^{1+\lambda}=o(\epsilon_n^\mathrm{bias})$ if and only if $(1+\lambda)(-(\varphi+1)\mathfrak r-\frac12(\psi-1))<-(\varphi+1)\mathfrak r$ (see~(\ref{eq256235})) which is by elementary calculations equivalent to $\psi>1-\frac{2\lambda}{1+\lambda}(\varphi+1)\mathfrak r$ which we assumed in~(\ref{eq:as_1}). Conversely, if the exponent is smaller than $-1$, then
$$
\frac 1{\bar b_n} \sum_{l=1}^n b_l \delta_l^{1+\lambda}=\mathcal O(n^{-(\rho+1)}\log n)
$$
and again $\frac 1{\bar b_n} \sum_{l=1}^n b_l \delta_l^{1+\lambda}=o(\epsilon_n^\mathrm{bias})$ since $-(\rho+1)<-(\varphi+1)\mathfrak r$ as a consequence of $\mathfrak r<\frac 12$ and the second inequality of~(\ref{eq:as_1}).

2)  \emph{Analysis of $\epsilon_n^\mathrm{bias}$.}   One has
\begin{align}\label{eq86255}
\sum_{k=1}^n b_k\delta_k^\mathrm{bias}=  \sum_{k=1}^n  k^\rho\, M^{-\alpha s_k} =\sum_{j=1}^{s_n-1} M^{-\alpha j} \sum_{k=y_j+1}^{y_{j+1}} k^\rho + \sum_{k=y_{s_n}+1}^ n k^\rho M^{-\alpha s_n},
\end{align}
where 
for $j\in\N$
$$
y_j:=\min\{n\in\N: \kappa_s \bar K_n^{\frac 1{2\alpha-\beta+1}}\geq M^j\}-1.
$$
Note that as $j\to\infty$
\begin{align}\label{eq94357}
y_j\sim \Bigl(\frac 1{\kappa_K}\bigl(\frac {M^j}{\kappa_s}\bigr)^{2\alpha-\beta+1}\Bigr)^{1/(\varphi+1)}\text{ \ and, hence,  \ } \frac 1{\rho+1} y_j^{\rho+1}\sim c_1 M ^{\mathfrak r_1j},  
\end{align}
where $\mathfrak r_1=\frac {\rho+1}{\varphi+1}(2\alpha-\beta+1)$ and  $c_1= \frac 1{\rho+1} \bigl(\frac 1{\kappa_K}\bigl(\frac {1}{\kappa_s}\bigr)^{2\alpha-\beta+1}\bigr)^{(\rho+1)/(\varphi+1)}$. Consequently,
$$
\sum_{r=y_j+1}^{y_{j+1}} k^\rho\sim \Bigl[\frac1{\rho+1} x^{\rho+1}\Bigr]_{y_j}^{y_{j+1}} \sim c_1
(M^{\mathfrak r_1}-1) M^{\mathfrak r_1 j}.
$$
As a consequence of~(\ref{eq:as_1}) one has $\mathfrak r_1 >\alpha$ and insertion of the above equivalences yields  as $n\to\infty$
\begin{align}\begin{split}\label{eq46748}
\sum_{j=1}^{s_n-1}M^{-\alpha j} \sum_{r=y_j+1}^{y_{j+1}} k^\rho &\sim c_1 (M^{\mathfrak r_1}-1)\sum_{j=1}^{s_n-1} M^{(\mathfrak r_1-\alpha)j}\\
&\sim  c_1 \underbrace{\frac{M^{\mathfrak r_1}-1}{M^{\mathfrak r_1-\alpha}-1}}_{=:c_2}  M^{(\mathfrak r_1-\alpha)s_n}.
\end{split}
\end{align}
Using again~(\ref{eq94357}) we obtain
\begin{align*}
\sum_{k=y_{s_n}+1}^n k^\rho M^{-\alpha s_n} \sim \frac 1{\rho+1} M^{-\alpha s_n} y_{s_n}^{\rho+1} \Bigl(\bigl(\frac n{y_{s_n}}\bigr)^{\rho+1}-1\Bigr)\sim c_1 M^{(\mathfrak r_1-\alpha) s_n} \Bigl(\bigl(\frac n{y_{s_n}}\bigr)^{\rho+1}-1\Bigr).
\end{align*}
We combine the previous estimate with the following consequence of~(\ref{eq94357}) and~(\ref{eq:M_sn})
\begin{align*}
y_{s_n}\sim \Bigl(\frac 1{\kappa_K} \bigl(\frac 1{\kappa_s} M^{s_n}\bigr)^{2\alpha-\beta+1}\Bigr)^{\frac 1{\varphi+1}}\sim n M^{-\frac {2\alpha-\beta+1}{\varphi+1}   \xi_n}
\end{align*}
and get with~(\ref{eq86255}), (\ref{eq46748}) and the definitions of $\mathfrak r_1, \mathfrak r, c_1$ and $\psi$  that
\begin{align*}
\sum_{k=1}^n b_k\delta_k^\mathrm{bias}&\sim   c_1  M^{(\mathfrak r_1-\alpha) s_n}   (c_2+M^{\frac {2\alpha-\beta+1}{\varphi+1}\xi_n}-1)\sim c_1(\kappa_s (\kappa_K n^{\varphi+1})^{\frac 1{2\alpha-\beta+1}} M^{-\xi_n})^{\mathfrak r_1-\alpha}
   (c_2+M^{\mathfrak r_1 \xi_n}-1)\\
   &=\frac1{\rho+1}\kappa_s^{-\alpha} \kappa_K^{-\frac \alpha{2\alpha-\beta-1}}  \psi_{\mathfrak r_1,-\alpha}(\xi_n) \, n^{-(\varphi+1)\mathfrak r+\rho+1}.
\end{align*}
Finally, we use that $\bar b_n\sim (\rho+1)^{-1} n^{\rho+1}$ to concude that
\begin{align*}
\epsilon_n^\mathrm{bias}&\sim \kappa_s^{-\alpha} \kappa_K^{-\mathfrak r}  \psi_{\mathfrak r_1,-\alpha}(\xi_n) \, n^{-(\varphi+1)\mathfrak r}=\eps_n^\mathrm{bias}.
\end{align*}

3) \emph{Analysis of $\epsilon_n^\mathrm{diff}$.}
We proceed similarly as in step two. Note that
\begin{align}\begin{split}\label{eq8234675}
\sum_{k=1}^n (b_k\delta_k^\mathrm{diff})^2&\sim \sum_{j=1}^{s_n-1} M^{(1-\beta)j} \sum_{k=y_{j}+1}^{y_{j+1}} k^{2\rho-\varphi}  +  M^{(1-\beta)s_n} \sum_{k=y_{s_n}+1}^ n k^{2\rho-\varphi}.
\end{split}\end{align}
Note that (\ref{eq:as_1}) implies that $2\rho-\varphi>-1$. By~(\ref{eq94357}) one has
$$
\frac 1{2\rho-\varphi+1} y_j^{2\rho-\varphi+1}\sim c_3 M^{(2\frac{\rho+1}{\varphi+1}-1)(2\alpha-\beta+1)j}\text{ \ for \ } c_3=\frac 1{2\rho-\varphi+1} \bigl(\frac 1{\kappa_K}\bigl(\frac {1}{\kappa_s}\bigr)^{2\alpha-\beta+1}\bigr)^{2\frac{\rho+1}{\varphi+1}-1}.
$$
As before we conclude that in terms of $\mathfrak r_2=(2\frac{\rho+1}{\varphi+1}-1)(2\alpha-\beta+1)$
$$
\sum_{k=y_j+1}^{y_{j+1}} k^{2\rho-\varphi}\sim \frac 1{2\rho-\varphi+1} y_j^{2\rho-\varphi+1}\bigl(\bigl( \frac {y_{j+1}}{y_j}\bigr)^{2\rho-\varphi+1}-1\bigr)\sim c_3 M^{\mathfrak r_2j}(M^{\mathfrak r_2}-1\bigr)
$$
and
$$
\sum_{k=y_{s_n}+1}^{n} k^{2\rho-\varphi}\sim c_3  M^{\mathfrak r_2 s_n}\bigl(\bigl(\frac n{y_{s_n}}\bigr)^{2\rho-\varphi+1} -1\bigr).
$$
With $c_4:=\frac{M^{\mathfrak r_2}-1}{M^{\mathfrak r_2+1-\beta}-1}$
we get that
\begin{align*}
  \sum_{j=1}^{s_n-1} M^{(1-\beta)j} \sum_{k=y_{j}+1}^{y_{j+1}} k^{2\rho-\varphi}\sim c_3 c_4 M^{(\mathfrak r_2+1-\beta)s_n}
  \end{align*}
We  insert the previous equation  into~(\ref{eq8234675}) and use equivalence~(\ref{eq:M_sn}) and the definitions of $\mathfrak r$, $c_3$ and $\psi$   to deduce that
\begin{align*}
\sum_{k=1}^n (b_k\delta_k^\mathrm{diff})^2&\sim c_3 M^{(\mathfrak r_2+1-\beta)s_n} \bigl( c_4+M^{\mathfrak r_2 \xi_n}-1\bigr)\\
&\sim   c_3 (\kappa_s (\kappa_K n^{\varphi+1})^{\frac1{2\alpha-\beta+1}}M^{-\xi_n})^{\mathfrak r_2 +1-\beta} \bigl( c_4+M^{\mathfrak r_2 \xi_n}-1\bigr)\\
&\sim  \frac 1{2\rho-\varphi+1}   \kappa_K^{\frac{1-\beta}{2\alpha-\beta+1}} \kappa_s^{1-\beta}\,\psi_{\mathfrak r_2, 1-\beta}(\xi_n)\, n^{2(\rho+1) -2(\varphi+1) \mathfrak r}
\end{align*}
Using that $\bar b_n\sim (\rho+1)^{-1} n^{\rho+1}$ we get that
\begin{align*}
\epsilon_n^\mathrm{diff}&\sim \frac{\rho+1}{\sqrt {2\rho-\varphi+1}} \kappa_K^{\frac 12\frac{1-\beta}{2\alpha-\beta+1}} \kappa_s^{\frac{1-\beta}2}\,\sqrt{\psi_{\mathfrak r_2, 1-\beta}(\xi_n)}\,  n^{-(\varphi+1)\mathfrak r}.
\end{align*}

4) \emph{Analysis of the cost  of the algorithm.} On $\{\theta_n\to\theta^*\}$ one has 
\begin{align*}
\mathrm{cost}_n&=\sum_{m=1}^n \sum_{k=1}^{s_m} N_k(s_m,K_m)\, C_k(\theta_{m-1})\\
&\sim \sum_{m=1}^n \sum_{k=1}^{s_m} \frac{K_m}{M^{s_m}} M^{\frac{\beta+1}2 (s_m-k)}\kappa_C M^k\\
&=\kappa_C \sum_{m=1}^n K_m\sum_{k=1}^{s_m} M^{-\frac {1-\beta}2(s_m-k)}\sim \frac{\kappa_C}{1-M^{-\frac{1-\beta}2}} \sum_{m=1}^n K_m\sim \frac{\kappa_C \kappa_K}{1-M^{-\frac{1-\beta}2}} n^{\varphi+1}. 
\end{align*}

5) \emph{Synthesis.} By Theorem~\ref{theo_RP_main} there exists an adapted sequence $(\vartheta_n)_{n\in\N}$ such that on $\{\theta_n\to\theta^*\}$ 
$$
(\eps_n^\mathrm{bias} )^{-1}(\vartheta_n-\theta^*)=H^{-1}\mu, \text{ in probability},
$$
and
$$
(\epsilon_n^\mathrm{diff})^{-1}(\bar \theta_n-\vartheta_n) \Rightarrow \mathcal N(0,H^{-1} \Gamma' (H^{-1})^\dagger).
$$
As we showed so far $\eps_n^\mathrm{bias}\sim \epsilon_n^\mathrm{bias}\approx \epsilon_n^\mathrm{diff}$ so that
$$
\vartheta_n-\theta^*-\eps_n^\mathrm{bias} H^{-1}\mu=o_P(\epsilon_n^\mathrm{bias})=o_P(\epsilon_n^\mathrm{diff}),
$$
where $o_P$ refers to small $o$ in probability. Hence, on $\{\theta_n\to\theta^*\}$
$$
(\epsilon_n^\mathrm{diff})^{-1}(\bar \theta_n-\vartheta_n)- (\epsilon_n^\mathrm{diff})^{-1}(\bar \theta_n-(\theta^*+\eps_n^\mathrm{bias} H^{-1}\mu)=o_P(1)
$$
so that
$$
 (\epsilon_n^\mathrm{diff})^{-1}(\bar \theta_n-(\theta^*+\eps_n^\mathrm{bias} H^{-1}\mu))\Rightarrow \mathcal N(0,H^{-1} \Gamma' (H^{-1})^\dagger)
$$
and by definition of $\Gamma'$
$$
\sqrt{\kappa_K(1-M^{-\frac{1-\beta}2})(\varphi+1)} (\epsilon_n^\mathrm{diff})^{-1}(\bar \theta_n-(\theta^*+\eps_n^\mathrm{bias} H^{-1}\mu))\Rightarrow \mathcal N(0,H^{-1} \Gamma (H^{-1})^\dagger).
$$
Now elementary computations show that the prefactor of $(\bar \theta_n-(\theta^*+\eps_n^\mathrm{bias} H^{-1}\mu))$ is equivalent to $(\eps_n^\mathrm{diff})^{-1}$.

\section{Proof of Theorem \ref{theo:main_pol_2}}\label{sec5}

We proceed similarly as in the proof of Theorem~\ref{theo:main_pol}. We note that the states $(\theta_n)_{n\in\N_0}$ of the algorithm satisfy equation~(\ref{dynsys2}) with 
$$
R_n=\begin{cases} \E[Z(\mathbb U_n;\theta_{n-1},s_n,K_n)|\cF_{n-1}]-f(\theta_{n-1}), & \text{ if }\theta_{n-1}\in D_{\theta^*},\\
Z(\mathbb U_n;\theta_{n-1},s_n,K_n)-f(\theta_{n-1}), &\text{ else},
\end{cases}
$$
and 
$$
D_n= \begin{cases} Z(\mathbb U_n;\theta_{n-1},s_n,K_n)- \E[Z(\mathbb U_n;\theta_{n-1},s_n,K_n)|\cF_{n-1}],&\text{ if }\theta_{n-1}\in B_{\eps'}(\theta^*),\\
0, &\text{ else},
\end{cases}
$$
where $D_{\theta^*}$ is the neighbourhood of $\theta^*$ appearing in Assumptions~(\ref{as_tail_1-1}) and (\ref{as_cost2_2}). We set 
$$\delta_n^\mathrm{bias} =n^{-\frac {\varphi+1}2} \sqrt{\log n}, \ \delta_n^\mathrm{diff}= (2\alpha \kappa_K)^{-1/2} n^{-\frac {\varphi}2}\sqrt{ \log_M n} \text{ \ and \ }\delta_n=n^{-\frac{\psi+\varphi}2} \sqrt{\log n}$$
and
$$
\epsilon_n^\mathrm{bias}=\bar b_n^{-1} \sum_{l=1}^n b_l \delta_l^\mathrm{bias}\text{ \ and \ }\epsilon_n^\mathrm{diff}=\bar b_n^{-1}\sqrt{ \sum_{l=1}^n (b_l \delta_l^\mathrm{diff})^2}.
$$

1) \emph{Verification of the assumptions of Theorem~\ref{theo_RP_main}.} 
Property (A) follows by elementary computations and property (B) by assumption. Note that
$$
b_k\delta_k^\mathrm{bias} \approx k^{\frac{2\rho-\varphi-1}2} \sqrt{\log k}\text{ \ and  \ } (b_k\delta_k^\mathrm{diff})^2 \approx k^{2\rho-\varphi} {\log k}
$$
and as consequence of the second assumption in~(\ref{eq:as_1-2}) the exponents $\frac{2\rho-\varphi-1}2$ and $2\rho-\varphi$ are strictly bigger than $-1$. Then property (C) follows easily by comparison with appropriate integrals.

Next we verify (D). By choice of $s_n$ one has $M^{-\alpha s_n}=o(\delta_n^\mathrm{bias})$. Hence, one has  on $\{\theta_n\to\theta^*\}$
$$
(\delta_n^\mathrm{bias})^{-1}R_n= \underbrace{(\delta_n^\mathrm{bias})^{-1} M^{-\alpha_{s_n} s_n}     }_{\to 0}  \underbrace{M^{\alpha_{s_n} s_n}  (\E[F_{s_n}(\theta,U)]\big|_{\theta=\theta_{n-1}}-f(\theta_{n-1}))}_{\text{eventually bounded}} \to 0
$$
as consequence of~(\ref{as_bias1_1-2}).  The second part of property (D) follows as in the proof of Theorem~\ref{theo:main_pol}.

Next we verify (E). Note that if $\theta_{n-1}$ attains a value in $B_{\eps'}(\theta^*)$, then
\begin{align}\begin{split}\label{eq835734-2}
 \cov(D_n|\cF_{n-1}) & = \cov(Z(\mathbb U_n; \theta_{n-1},s_n,K_n)|\cF_{n-1})\\
&= \sum_{r=1}^{s_n} \frac 1{N_r(K_n)} \,\cov(F_k(U,\theta)-F_{k-1}(U,\theta))\big|_{\theta=\theta_{n-1}}.
\end{split}\end{align}
One  has $\log_M (M^{-s_n} K_n) =-s_n+\varphi \log_M n +\mathcal O(1)$ and  $s_n\sim \frac{\varphi+1}{2\alpha}\log_M n$ by choice of $s_n$. The first assumption in~(\ref{eq:as_1-2}) implies that $\frac{\varphi+1}{2\alpha}<\varphi$ so that $\lim_{n\to\infty} M^{-s_n} K_n=\infty$. Hence, as consequence of~(\ref{as_var1_1-2}) one has on $\{\theta_n\to\theta^*\}$ as $n,r\to\infty$ with $r\leq s_n$
$$
K_n\frac 1{N_r(K_n)} \cov(F_k(U,\theta)-F_{k-1}(U,\theta))\big|_{\theta=\theta_{n-1}}\to \Gamma.
$$
Furthermore there exists a constant $C$ such that for every $n\in\N$ and $\theta\in D_{\theta^*}$
$$
\bigl\|K_n\frac 1{N_r(K_n)} \cov(F_k(U,\theta)-F_{k-1}(U,\theta))\bigr\|\leq C.
$$
Altogether we thus get with~(\ref{eq835734-2}) that on $\{\theta_n\to\theta^*\}$
$$
 \cov(D_n|\cF_{n-1})\sim \frac {s_n}{K_n} \,\Gamma\sim (\delta_n^\mathrm{diff})^2 \,\Gamma.
$$
The second part of (E) is shown in complete analogy to the proof of Theorem~\ref{theo:main_pol}.

To verify the third part we notice that as long as $B(\theta^*,\eps')\subset D_{\theta^*}$ one has
\begin{align*}
\E[\1_{\{\theta_{n-1}\in B(\theta^*,\eps'')\}}|D_n|^2]\leq C_{\mathrm{TAIL}} \sum _{l=1}^{s_n} \frac 1{N_l(K_n)} M^{-l}\approx (\delta_n^\mathrm{diff})^2
\end{align*}
so that using that $\delta_n^\mathrm{diff}\approx n^{-\varphi/2} \sqrt{\log n}$
\begin{align*}
\sqrt{\gamma_n}\,\E[\1_{\{\theta_{n-1}\in B(\theta^*,\eps'')\}}|D_n|^2]^{1/2}=\mathcal O\bigl(n^{-\frac\psi2-(\varphi+1)\mathfrak r+\frac 12}\bigr)=\mathcal O(\delta_n).
\end{align*}

We verify (F). We distinguish two cases. If $\rho-(1+\lambda)\frac{\varphi+\psi}2>-1$ one has
$$
\frac 1{\bar b_n} \sum_{k=1}^n b_k \delta_k^{1+\lambda}\approx \frac 1{n^{\rho+1}} \sum_{k=1}^n k^{\rho-(1+\lambda)\frac{\varphi+\psi}2} (\log k)^{(1+\lambda)/2} \approx n^{-(1+\lambda)\frac{\varphi+\psi}2} (\log n)^{(1+\lambda)/2} ,
$$
which is of oder $o(\eps_n^\mathrm{diff})$ since $\eps_n^\mathrm{diff}\approx n^{-\frac{\varphi+1}2}\sqrt{\log n}$ and $\frac{\psi+\varphi}2(1+\lambda)>\frac {\varphi+1}2$ as a consequence of the third assumption of~(\ref{eq:as_1-2}). If $\rho-(1+\lambda)\frac{\varphi+\psi}2\leq -1$, then
$$
\frac 1{\bar b_n} \sum_{l=1}^n b_l \delta_l^{1+\lambda}\mathcal O\Bigl( \frac 1{n^{\rho+1}} \sum_{l=1}^n l^{-1} (\log l)^{(1+\lambda)/2}\Bigr) =\mathcal O \bigl( n^{-(\rho+1)}  (\log n)^{(3+\lambda)/2} \bigr)
$$
which is of order $o(\eps_n^\mathrm{diff})$ since $\rho+1>\frac {\varphi+1}2$ by the second assumption of of~(\ref{eq:as_1-2}).

2) \emph{The asymptotics of $\epsilon^\mathrm{bias}_n$ and $\epsilon_n^\mathrm{diff}$.}
One has
\begin{align*}
\epsilon_n^\mathrm{bias} \sim (\rho+1) n^{-(\rho+1)} \sum_{k=1}^n k^{\rho-\frac{\varphi+1}2}\sqrt {\log k} \sim \frac {2(\rho+1)}{2\rho-\varphi+1} n^{-\frac{\varphi+1}2}\sqrt {\log n},
\end{align*}
where we used that the exponent $\rho-\frac{\varphi+1}2$ is strictly bigger than $-1$ as observed in the verification of (C).

  Note that using that $2\rho-\varphi>-1$ as consequence of the second assumption in~(\ref{eq:as_1-2}) we get that
 \begin{align}\begin{split}\label{eq8234675-2}
\sum_{k=1}^n (b_k\delta_k^\mathrm{diff})^2&= \sum_{k=1}^{n} k^{2\rho} (2\alpha\kappa_K)^{-1} k^{-\varphi} \log_M k\\
&\sim (2\alpha \kappa_K (2\rho-\varphi+1))^{-1} n^{2\rho-\varphi+1} \log_Mn.
\end{split}\end{align}
Consequently,
\begin{align}\begin{split}\label{eq78943556}
(\epsilon_n^\mathrm{diff})^2&=\frac 1{\bar b_n^2} \sum_{k=1}^n (b_k \delta_k^\mathrm{diff})^2 \sim (\rho+1)^2(2\alpha \kappa_K (2\rho-\varphi+1))^{-1}n^{-(\varphi+1)} \log_Mn\\
&=\frac 1{2\alpha\kappa_K} \frac {\bigl(\frac{\rho+1}{\varphi+1}\bigr)^2}{2\frac {\rho+1}{\varphi+1}-1} n^{-(\varphi+1)} \log_Mn^{\varphi+1}=(\eps_n^\mathrm{diff})^2.
\end{split}\end{align}

3) \emph{Analysis of the cost  of the algorithm.}  As a consequece of~(\ref{as_cost_2}) and~(\ref{as_cost2_2}) and the choice of $K_n$ and $s_n$, one has on $\{\theta_n\to\theta^*\}$ 
$$
 \sum_{k=1}^{s_n} N_k(K_n)\, C_k(\theta_{n-1})\sim  \kappa_C s_n K_n\sim \kappa_C \kappa_K  \alpha^{-1}(\varphi+1)n^\varphi \log_M n^{\frac{\varphi+1}2}
$$
so that
\begin{align*}
\mathrm{cost}_n&=\sum_{m=1}^n \sum_{k=1}^{s_m} N_k(K_m)\, C_k(\theta_{m-1}) \sim \kappa_C \kappa_K  \alpha^{-1}(\varphi+1) \sum_{m=1}^n  m^\varphi \log_M m^{\frac{\varphi+1}2}\\
&\sim \kappa_C \kappa_K  \alpha^{-1} n^{\varphi+1} \log_M n^{\frac{\varphi+1}2}. 
\end{align*}
On $\{\theta_n\to\theta^*\}$ one has
$$
\log_M \mathrm{cost_n}\sim \log_M n^{\varphi+1}
\text{ \ and thus \ }
n^{-(\varphi+1)} \sim \frac{\kappa_C \kappa_K}{2\alpha} \frac{\log_M \mathrm{cost}_n}{\mathrm{cost_n}}.
$$
Consequently, on $\{\theta_n\to\theta^*\}$
$$
\eps_n^\mathrm{diff} \sim  \frac{\sqrt{\kappa_C}}{2\alpha} \frac {\frac{\rho+1}{\varphi+1}}{\sqrt{2\frac{\rho+1}{\varphi+1}-1}}  \,\frac{\log_M \mathrm{cost}_n}{\sqrt{\mathrm{cost}_n}}. 
$$

4)  \emph{Synthesis.}  We showed that $\epsilon_n^\mathrm{bias}\approx\epsilon_n^\mathrm{diff}$. As a consequence of Theorem~\ref{theo_RP_main} we have on $\{\theta_n\to\theta^*\}$
$$
(\epsilon_n^\mathrm{diff})^{-1} (\bar\theta_n-\theta^*) \Rightarrow \mathcal N(0,H^{-1}\Gamma(H^{-1})^\dagger)
$$
and due to~(\ref{eq78943556}) we can replace $\epsilon_n^\mathrm{diff}$ by $\eps_n^\mathrm{bias}$. The reformulation of $\eps_n^\mathrm{diff}$ in terms of $\mathrm{cost}_n$ was already derived in part 3).

\begin{appendix}

\section{$L^2$-error bounds}

\begin{theorem}\label{thm1} 
Suppose that  $(\theta_n)_{n\in\N_0}$ satisfies the recursion~(\ref{dynsys2}) and suppose that $\theta^*$ is an $L$-contracting zero of $f$. 
Let $(\gamma_n)_{n\in\N}$ and $(\delta_n)_{n\in\N}$  sequences of strictly positive reals with $\sum_{n=1}^\infty \gamma_n=\infty$ and
\begin{align}\label{eq23478}
\limsup_{n\to\infty} \frac{1}{\gamma_n}\,\frac{\delta_{n-1} - \delta_{n}}{\delta_{n-1}} \leq L.
\end{align}
Further suppose that there exists $\eps'\in(0,\infty)$ with
$$
\limsup_{n\to\infty} \delta_n^{-1} \E\bigl[ \1_{ B_{\eps'}(\theta^*)} (\theta_{n-1}) |R_n |^2\bigr]^{1/2}<\infty
$$
and
$$
\limsup_{n\to\infty}\Bigl(\frac{\delta_n}{\sqrt{\gamma_n}}\Bigr)^{-1} \, \E[\1_{ B_{\eps'}(\theta^*)} (\theta_{n-1})
 |D_n|^2]^{1/2}<\infty.
$$
Then there exist $\eps,C \in (0,\infty)$  such that for all $n_0\in \N$,
\begin{align}\label{eq893467}
 \limsup_{n\to\infty} \delta_n^{-1} \E\bigl[\1_{\{\theta_m\in B_\eps(\theta^*)\text{ for } m=n_0,\dots,n-1\}} |\theta_n-\theta^*|^2\bigr]^{1/2}\leq C.
\end{align}
\end{theorem}

For completeness we provide a proof.

\begin{proof}
Since the statement of the theorem is translation invariant we can and will assume that $\theta^*=0$.
Since $H$ is  $L$-contracting  there  exists an inner product $\langle\cdot,\cdot\rangle$ on $\R^d$ and $\eps_0>0$ such that for sufficiently large $n$ 
$$
\|I+\gamma_n H\|\leq 1-\gamma_n (L+\eps_0)
$$
for the induced matrix norm. Further there exist finite strictly positive constants $C_1,C_2,C_3 ,\eps'$ such that for sufficiently large $n$
\begin{align}\label{eq87435}
 \E[\1_{ B_{\eps'}(0)} (\theta_{n-1}) \|D_n\|^2]\leq C_1 \frac{\delta_n^2}{\gamma_n}, \qquad 
   \E[\1_{B_{\eps'}(0)} (\theta_{n-1} ) \|R_n\|^2]\leq C_2 \delta^2_n
\end{align}
and as a consequence of~(\ref{eq23478})
\begin{align}\label{eq87436}
\frac {\delta_{n-1}}{\delta_n} \leq C_3.
\end{align}
Moreover, assuming that $\eps'>0$ is chosen sufficiently small one has for all $\theta\in B_{\eps'}(0)$
$$
\|f(\theta)-H\theta\|\leq \eps_1 \|\theta\|
$$
where $\eps_1:= \frac 1{20}C_3^{-2}\eps_0$.\smallskip

We fix $n_0\in\N$ arbitrarily and consider for $n\geq n_0$
$$\varphi_n=\delta_n^{-2} \E[\1_{\ID_n}\|\theta_n\|^2],$$
where $\ID_n:=\{\forall l=n_0,\dots,n-1: \|\theta_l\|\leq \eps'\}$.
One has  for $n\geq n_0$
\begin{align*}
\varphi_n&\leq \delta_n^{-2} \E[\1_{\ID_{n-1}} \|\theta_{n-1} +\gamma_n(f(\theta_{n-1})+R_n+D_n)\|^2]\\
&= \underbrace{ \delta_n^{-2}  \E[\1_{\ID_{n-1}} \|\theta_{n-1} +\gamma_n(f(\theta_{n-1})+R_n)\|^2]}_{=:I_1(n)}+\underbrace{ \frac {\gamma_n^2}{\delta_n^2} \E[\1_{\ID_{n-1}}\|D_n\|^2]}_{=:I_2(n)}.
\end{align*}
By assumption the second summand satisfies $I_2(n)\leq  C_1 \gamma_n$ for sufficiently large $n$.
In order to estimate the first summand $I_1(n)$ we first note that for sufficiently large $n$
$$
\frac {\delta_{n-1}}{\delta_n}=1+\frac{\delta_{n-1}-\delta_n}{\delta_n}\leq 1+ (L+\eps_0/2)\gamma_n,
$$
$$
 (1-\gamma_n(L+\eps_0))^2(1+(L+\eps_0/2)\gamma_n)^2=1-\eps_0\gamma_n +o(\gamma_n)\leq 1-\eps_0
$$
and
$$
 \E[\1_{\ID_{n-1}} \|\theta_{n-1}\| \|R_n\|]\leq \eps_1  \E[\1_{\ID_{n-1}} \|\theta_{n-1}\|^2]+\frac 1{\eps_1} \E[\1_{\ID_{n-1}} \|R_n\|^2].
$$
Consequently, we get with the triangle inequality that  for sufficiently large $n$ 
\begin{align*}
 I_1(n)&=\delta_n^{-2}  \E\bigl[\1_{\ID_{n-1}}\bigl( \|\theta_{n-1}+\gamma_n H\theta_{n-1}\|^2  +2\gamma_n  \langle \theta_{n-1}+\gamma_n H\theta_{n-1}, f(\theta_{n-1})-H\theta_{n-1} +R_n\rangle \\
 &\qquad\qquad + \gamma_n^2 \|f(\theta_{n-1})-H\theta_{n-1} +R_n \|^2\bigr)\Bigr]\\
 & \leq  (1-\gamma_n(L+\eps_0))^2(1+(L+\eps_0/2)\gamma_n)^2 \varphi_{n-1} \\
 &\qquad +2\gamma_n \delta_n^{-2} \E[\1_{\ID_{n-1}}  \|\theta_{n-1}\| (\|f(\theta_{n-1})-H\theta_{n-1}\|+\|R_n\|)]\\
 &\qquad +2\gamma_n^2 \delta_n^{-2}  \E[\1_{\ID_{n-1}} (\|f(\theta_{n-1})-H\theta_{n-1}\|^2+\|R_n\|^2)] \\
 &\leq (1-\sfrac12 \eps_0\gamma_n)\varphi_{n-1}\\
 &\qquad +2\gamma_n\delta_n^{-2}\bigl( 2\eps_1 \E[\1_{\ID_{n-1}}\|\theta_{n-1}\|^2] +\frac 1{\eps_1}  \E[\1_{\ID_{n-1}}\|R_n\|^2]\bigr)\\
 &\qquad + 2\gamma_n^2 \delta_n^{-2}  \E[\1_{\ID_{n-1}} (\|\theta_{n-1}\|^2+\|R_n\|^2)] .
\end{align*}
As a consequence of~\eqref{eq87435} and~\eqref{eq87436} one has
$$\delta_n^{-2}  \E[\1_{\ID_{n-1}}\|\theta_{n-1}\|^2]\leq C_3^2 \varphi_{n-1}  
\text{ \ and \ }\delta_n^{-2}  \E[\1_{\ID_{n-1}}\|R_{n}\|^2]\leq {C_2}$$ and we thus obtain that for sufficiently large $n$
\begin{align*}
I_1(n)&\leq  (1-\sfrac12 \eps_0\gamma_n)\varphi_{n-1}+ 4\gamma_n \eps_1C_3^2 \varphi_{n-1} +\frac{2\gamma_n}{\eps_1} C_2+ 2\gamma_n^2C_3^2\varphi_{n-1}+2\gamma_n^2C_2\\
&=  \bigl(1- (\sfrac12 \eps_0-4\eps_1 C_3^2-2\gamma_nC_3^2)\gamma_n\bigr)\varphi_{n-1}+\bigl(\frac{2}{\eps_1} C_2+2\gamma_nC_2\bigr)\gamma_n.
\end{align*}
By choice of $\eps_1$ we have that $\sfrac12 \eps_0-4\eps_1 C_3^2-2\gamma_nC_3^2\geq \frac 14 \eps_0$ for sufficiently large $n$. 
Combining this with the estimate for $I_2(n)$ we thus get that for sufficiently large $n$
$$
\varphi_n\leq (1-\sfrac14 \eps_0 \gamma_n)\varphi_{n-1}+ \kappa \gamma_n
$$
where $\kappa:=\frac{4}{\eps_1} C_2$. We rewrite the inequality as
$$
\varphi_n-\sfrac{4\kappa}{\eps_0} \leq (1-\sfrac14 \eps_0 \gamma_n)(\varphi_{n-1}-\sfrac{4\kappa}{\eps_0})
$$This implies that for a sufficiently large $n_1\geq n_0$ one has for all $n\geq n_1$
$$
\varphi_n -\sfrac{4\kappa}{\eps_0} \leq \bigl(\varphi_{n_1}-\sfrac{4\kappa}{\eps_0}\bigr) \prod_{l=n_1+1}^n (1-\sfrac14 \eps_0 \gamma_l)\to 0,
$$
where convergence holds since $\sum_{l=n_1+1}^\infty \gamma_l=\infty$. It follows that 
$$
\limsup_{n\to\infty}\varphi_n\leq \sfrac {4\kappa}{\eps_0}
$$
with the constant on the right hand side not depending on the choice of $n_0$.
\end{proof}

\end{appendix}

\bibliographystyle{plain}
\bibliography{stoch_approx}

\end{document}